\documentclass{amsart}

\usepackage{amssymb,latexsym,amsmath,extarrows,amsthm, tikz}
\usepackage{graphicx}
\usepackage{xcolor}
\usepackage{mathabx}

\usepackage[symbol]{footmisc}

\newtheorem*{theorem*}{Theorem}
\newtheorem*{acknowledgement}{Acknowledgements}
\newtheorem{theorem}{Theorem}[section]
\newtheorem{lemma}[theorem]{Lemma}

\newtheorem{corollary}[theorem]{Corollary}

\begin{document}

\title[On the multiparameter Falconer distance problem]{On the multiparameter Falconer distance problem}
%\author{Xiumin Du, Yumeng Ou, and Ruixiang Zhang}
\author[X. Du]{Xiumin Du}
\author[Y. Ou]{Yumeng Ou}
\author[R. Zhang]{Ruixiang Zhang}

\address[X. Du]{Department of Mathematics, Northwestern University, Evanston, IL 60208}
\address[Y. Ou]{Department of Mathematics, University of Pennsylvania, Philadelphia, PA 19104}
\address[R. Zhang]{Department of Mathematics, University of California, Berkeley, Berkeley, CA 94720}

%\centerline{\today}
\begin{abstract} 
We study an extension of the Falconer distance problem in the multiparameter setting. Given $\ell\geq 1$ and $\mathbb{R}^{d}=\mathbb{R}^{d_1}\times\cdots \times\mathbb{R}^{d_\ell}$, $d_i\geq 2$. For any compact set $E\subset \mathbb{R}^{d}$ with Hausdorff dimension larger than $d-\frac{\min(d_i)}{2}+\frac{1}{4}$ if $\min(d_i) $ is even, $d-\frac{\min(d_i)}{2}+\frac{1}{4}+\frac{1}{4\min(d_i)}$ if $\min(d_i) $ is odd, we prove that the multiparameter distance set of $E$ has positive $\ell$-dimensional Lebesgue measure. A key ingredient in the proof is a new multiparameter radial projection theorem for fractal measures.
\end{abstract}

\maketitle

%%%%%%%%%%%%%%%%%%%%%%%%%%%%
\section{introduction}

%%%%%%%%% chain intro
Let $\vec{d}=(d_1,\cdots,d_\ell)$ and $d=d_1+\cdots + d_\ell$ with integers $\ell\geq 1$ and $d_i\geq 2$, $\forall 1\leq i\leq \ell$. Denote $x=(x_1,\cdots,x_\ell)\in \mathbb{R}^{d_1}\times\cdots\times\mathbb{R}^{d_\ell}=\mathbb{R}^d$. For a compact set $E \subset \mathbb{R}^d$, define its multiparameter distance set to be
\smallskip 
\[
\Delta^{\vec{d}}(E):=\{(|x_1-y_1|,\ldots, |x_\ell-y_\ell|)\in \mathbb{R}^{\ell}:\, x, y \in E\}.
\]

We are interested in studying how large the Hausdorff dimension of $E$ needs to be in order to ensure that $|\Delta^{\vec{d}}(E)|_\ell$, the Lebesgue measure of $\Delta^{\vec{d}}(E)$ in $\mathbb{R}^{\ell}$, is positive. In particular, when $\ell=1$, this is precisely the Falconer distance problem, for which the conjecture \cite{Falc86} (still open in all dimensions $d\geq 2$) is that
\[
\text{dim}(E)>\frac{d}{2}\implies |\Delta(E)|_1>0,
\]where we have denoted in this case $\Delta(E):=\Delta^{\vec{d}}(E)$ for short. Here and throughout the article, ${\rm dim}$ denotes the Hausdorff dimension, and oftentimes we write $|\cdot|=|\cdot|_\ell$ when the dimension of the Lebesgue measure is clear from the context. 

The Falconer distance conjecture, which is a famous difficult problem in geometric measure theory, is a continuous version of the celebrated Erd\H{o}s distinct distance conjecture whose two-dimensional case was resolved by Guth and Katz \cite{GK15}. The study of the Falconer problem is naturally related to Fourier restriction theory, projection theory of fractal measures, and incidence geometry. It has attracted a great amount of attention over the decades (for instance \cite{Mat85, B94, W99, Erd05}) and has seen some very recent breakthroughs. See \cite{GIOW, DIOWZ, DGOWWZ18, DZ18} and the references therein for more details.

As far as we know, the multiparameter version of the Falconer distance problem was first proposed by Hambrook--Iosevich--Rice \cite{HIR}, where the authors utilized a group action method to turn the original question into the estimate of an integral that resembles the Mattila integral \cite{Mat85} for the original Falconer problem. Moreover, they observed that, by considering the construction of Falconer \cite{Falc86} in one hyperplane crossed with full boxes in the other hyperplanes, for each $0<s<d-\frac{d_{\min}}{2}$, where $d_{\min}:=\min (d_1,\cdots, d_\ell)$, there exists a compact set $E$ such that $\text{dim}(E)=s$ and $|\Delta^{\vec{d}}(E)| = 0$. It is unclear, however, whether one can show that the multiparameter Mattila integral is bounded using spherical decay estimate of Fourier transform of measures, as in the one-parameter case \cite{HIR per}.

In the discrete setting, multiparameter Falconer-Erd\H{o}s type distance problems have also been studied both in the Euclidean spaces and finite fields, see \cite{BI17, IJP17, FMPS21}.

In this article, we obtain the first result for the multiparameter Falconer distance problem in the continuous setting, towards the conjectured dimensional threshold $d-\frac{d_{\min}}{2}$. Our proof is based on a multiparameter extension of some key ideas arising in recent works \cite{GIOW, DIOWZ} on the original Falconer distance problem, and in fact implies the stronger pinned distance result. In the one-parameter case, these ideas have proven to be useful in not only estimating the dimension of distance set \cite{Liu19, Shm20} but also the study of restricted families of projections \cite{Harris19}. Hence the multiparameter version of the framework obtained here is expected to be of independent interest and may have further applications in other problems in geometric measure theory that display non-isotropic dilation structure.

\begin{theorem}\label{thm: main1}
Let $d=d_1+\cdots +d_\ell$ with $d_i\geq 2$, $\ell\geq 1$ and denote $d_{\min}=\min(d_1,\cdots, d_\ell)$. Then, for any compact set $E\subset \mathbb{R}^d$ satisfying 
\[
\text{dim}(E)>\begin{cases} d-\frac{d_{\min}}{2}+\frac{1}{4},& d_{\min} \text{ is even},\\ d-\frac{d_{\min}}{2}+\frac{1}{4}+\frac{1}{4d_{\min}},& d_{\min} \text{ is odd},\end{cases}
\]there exists $x\in E$ such that
\[
\left|\Delta^{\vec{d}}_x(E):=\{(|x_1-y_1|,\ldots, |x_\ell-y_\ell|)\in \mathbb{R}^{\ell}:\, y \in E\}\right|>0.
\]
\end{theorem}

The assumption $d_i\geq 2$ in the multiparameter Falconer problem is necessary. In fact, if there is some $d_i=1$, then there are examples showing that in such cases no nontrivial threshold can be obtained. To see this, assume that $d_1=1$, then there exists Cantor type set $E_1\subset \mathbb{R}^{d_1}$ with arbitrarily large dimension such that $|\Delta(E_1)|_1=0$. By considering the set $E=E_1\times B_1^{d_2}\times\cdots\times B_1^{d_\ell}$, where $B_1^{d_i}$ denotes the unit ball centered at the origin in $\mathbb{R}^{d_i}$, one can construct set in $\mathbb{R}^d$ with arbitrarily large dimension but satisfies $|\Delta^{\vec{d}}(E)|_{\ell}=0$.

In the case that $\ell=1$, Theorem \ref{thm: main1} recovers the best known result towards the Falconer conjecture when $d$ is even, originally proved in \cite{GIOW, DIOWZ}. When $d$ is odd, the conclusion given by Theorem \ref{thm: main1} is inferior to the state-of-the-art result (dimensional threshold $\frac{d}{2}+\frac{1}{4}+\frac{1}{8d-4}$) towards the Falconer conjecture, proved in \cite{DGOWWZ18, DZ18}. Both of these two works are based on Mattila's framework which reduces the Falconer problem to the spherical decay estimate of Fourier transform of measures \cite{Mat85}, and it remains to be understood whether a similar reduction can be achieved via the multiparameter Mattila integral derived in \cite{HIR}.

In general, if the threshold $d-\frac{d_{\min}}{2}+\delta_0$ can be obtained for the multiparameter problem for some $\ell\geq 2$ (in other words, 
for all compact $\tilde{E}\subset \mathbb{R}^d$, ${\rm dim}(\tilde{E})>d-\frac{d_{\min}}{2}+\delta_0$ implies that $|\Delta^{\vec{d}}(\tilde{E})|_{\ell}>0$), then, the threshold $\frac{d}{2}+\delta_0$ can be obtained for the original Falconer problem, by considering the set $\tilde{E}:=E\times B_1^{d}\times\cdots\times B_1^d\subset \mathbb{R}^{d\ell}$. Indeed, it is direct to see that $|\Delta(E)|_1>0$ if and only if $|\Delta^{\vec{d}}(\tilde{E})|_\ell>0$, and ${\rm dim}(E)>\frac{d}{2}+\delta_0$ implies that ${\rm dim}(\tilde{E})>d\ell-\frac{d}{2}+\delta_0$.

Compared to the original one-parameter distance set, the multiparameter distance set captures the non-isotropic geometric information of the set. For instance, consider two compact subsets in $\mathbb{R}^2\times\mathbb{R}^2$: $E_1=[0,1]^2\times \{(0,0)\}$ and $E_2=[0,1]^2\times [0,1]^2$. Apparently, both $|\Delta(E_1)|_1$, $|\Delta(E_2)|_1$ are positive. However, one can see that $|\Delta^{(2,2)}(E_1)|_2=0$ while $|\Delta^{(2,2)}(E_2)|_2>0$.

In addition, we point out that the multiparameter distance problem in the continuous setting is of a genuine multiparameter nature, and doesn't seem to be approachable using a tensor product type argument. This is significantly different from the discrete setting, where one can use distance results with smaller number of parameters as blackboxes to study distance problems with more parameters. Indeed, this is exactly the path taken in \cite{IJP17}, where the authors applied the Guth--Katz distinct distance result as a blackbox in the study of multiparameter distinct distance problems.

At first sight, such a strategy might seem to work in the continuous case as well. For example, given a set $E\subset \mathbb{R}^{2}\times \mathbb{R}^2$ with ${\rm dim}(E)>4-\frac{2}{2}+\frac{1}{4}=3+\frac{1}{4}$, define $E_{x}=\{y\in \mathbb{R}^2:\, (x,y)\in E\}$, $\forall x\in \mathbb{R}^2$. We call $E_x$ a good fiber if it satisfies ${\rm dim}(E_x)>\frac{5}{4}$. Suppose one could show that $\{x:\, E_x \text{ good}\}$ has Hausdorff dimension larger than $\frac{5}{4}$, then a Fubini type argument, combined with the two dimensional Falconer result (dimensional threshold $\frac{5}{4}$) of \cite{GIOW}, would imply that $|\Delta^{(2,2)}(E)|_2>0$. However, this doesn't hold true in general. In fact, there exists example of a compact set with Hausdorff dimension $2$ in the plane that is a graph over an uncountable set of directions (see \cite{DF} for such a construction). One can use this to easily build a set $E\subset \mathbb{R}^2\times \mathbb{R}^2$ with arbitrarily large dimension that doesn't satisfy the good fiber property. This seems to be a manifestation of the fact that Hausdorff dimension does not always behave well when forming Cartesian products, and is a key difference between the continuous and discrete versions of the multiparameter distance problem.

\subsection{Strategy and new difficulties in the multiparameter setting}

The strategy of the proof of Theorem \ref{thm: main1} is inspired largely by \cite{GIOW, DIOWZ}, where the authors studied the original Falconer conjecture (i.e. $\ell=1$) in even dimensions. Here is a sketch of the main framework. Fix a set $E$ with dimension larger than $\alpha$, one starts with a choice of two disjoint subsets $E_1, E_2\subset E$ each of which supports a Frostman measure of exponent $\alpha$, denoted by $\mu_1, \mu_2$ respectively. Then, in Step 1, one prunes the measure $\mu_1$ to get a new complex valued measure $\mu_{1,g}$ and shows that their $L^1$ error is small using a radial projection argument; in Step 2, via a weighted restriction estimate, one shows that an $L^2$ quantity involving $\mu_{1,g}$ and $\mu_2$ is finite by using a refined decoupling estimate proved in \cite{GIOW}, which then implies the desired result through an identity of Liu \cite{Liu18}. A particularly nice feature of this strategy is that it deduces the stronger pinned distance result.

Unfortunately, neither of the two steps described above can be iterated directly in the multiparameter setting, mainly because the slices of a Frostman measure are in general not necessarily still Frostman measures satisfying desirable energy estimates. More precisely, an iteration of the pruning process of the Frostman measure in Step 1 doesn't seem to produce a measure that is sufficiently nice. And since the weight function, generated by the Frostman measure, is not necessarily a tensor product, the weighted restriction estimate in Step 2 doesn't follow from iterating the one-parameter argument. In fact, even without the presence of the weight, one doesn't seem to be able to iterate the refined decoupling inequality in Step 2. The reason is that the application of the refined decoupling requires a dyadic pigeonholing process, which, when performed in each variable separately in the multiparameter setting, would yield a loss much worse than $\log$.

In this article, we overcome the difficulties in Step 1 by studying multiparameter versions of radial projections of fractal measures. In particular, we extend the radial projection theorem of Orponen \cite{O17b} to the multiparameter setting, which allows us to identify an effective way to prune the Frostman measure in Step 1. More precisely, the multiparameter radial projection theorem will be used to remove product bad wave packets from the Frostman measure. To the best of our knowledge, such a theorem seems to be new, and is expected to be of independent interest. See below for further discussions on the theorem and sample applications to the visibility of sets. In addition, we prove new multiparameter weighted restriction estimate to tackle the difficulties arising in Step 2, which relies on a newly observed multiparameter refined decoupling theorem. 

%In addition, it seems that the multiparameter Falconer problem is not rotation invariant. {\color{blue}{(Alex mentioned that the lattice example may work).}} Propose the rotation invariant version of the question?

\subsection{Multiparameter radial projection and application to visibility of sets}

As the statement of the full version of the multiparameter radial projection theorem (Theorem \ref{thm: multipara proj}) is quite technical, here we only state a special case of it (``the large $\alpha$ case''). 

We first introduce some notation. Let $d=d_1+\cdots d_\ell$ and $d_i\geq 2$, $\ell\geq 1$. Denote by $\pi_i:\mathbb{R}^d\to\mathbb{R}^{d_i}$ the orthogonal projection. Given $y=(y_1,\cdots,y_\ell)\in \mathbb{R}^{d_1}\times\cdots\times\mathbb{R}^{d_\ell}=\mathbb{R}^d$, define the $\ell$-parameter radial projection $P^{(\ell)}_y: \mathbb{R}^d \setminus \{x:x_i=y_i \text{ for some } i=1,\cdots,\ell\}\to S^{d_1-1}\times \cdots \times S^{d_\ell-1}$ by
$$
P^{(\ell)}_y (x) = \left(\frac{x_1-y_1}{|x_1-y_1|},\cdots, \frac{x_\ell-y_\ell}{|x_\ell-y_\ell|}\right)\,.
$$ In the following, $\mathcal{M}(\mathbb{R}^d)$ denotes the space of compactly supported Radon measures on $\mathbb{R}^d$.

\begin{theorem}\label{thm: multipara proj0}
For every $\alpha>d-1$ and $\beta>2(d-1)-\alpha$, there exists $p=p(\alpha,\beta)>1$ such that the following holds. Suppose that $\mu, \nu \in \mathcal{M}(\mathbb{R}^d)$ satisfy 
\begin{enumerate}
\item $\pi_i({\rm supp}\,\mu) \cap \pi_i({\rm supp}\,\nu) = \varnothing$, $\forall i=1,\ldots, \ell$;
\item $\mu(B(x,r))\leq  C_\alpha(\mu) r^\alpha$, $\nu(B(x,r))\leq C_\beta(\nu)r^\beta$, $\forall x\in \mathbb{R}^d$, $\forall r>0$. 
\end{enumerate}
Then,  
\[
\int \|P^{(\ell)}_y\mu\|^p_{L^p(S^{d_1-1}\times\cdots\times S^{d_{\ell}-1})}\,d\nu(y) \lesssim C_\alpha(\mu)^p C_\beta(\nu),
\]
where the implicit constant depends only on $d$ and the diameter of ${\rm supp}\,\mu \cup {\rm supp}\,\nu$.
\end{theorem}

As a remark, when $\ell=1$, Theorem \ref{thm: multipara proj0} is simply a weaker version of Orponen's radial projection theorem \cite{O17b}, where the ball condition (2) in the above is replaced by a weaker energy condition for the measures. This stronger ball condition is used in an essential way in our argument to estimate a multiparameter analogue of the energy of a measure. 

Theorem \ref{thm: multipara proj0} is sharp up to the endpoint, in the sense that the assumption $\alpha>d-1$ cannot be further relaxed.  Indeed, this is because the condition is already required in the one-parameter case. To see the sharpness, consider the following $\ell$-parameter example: $\mu=\mu_1\times m_2\chi_{B_1^{d_2}}\times\cdots\times m_\ell\chi_{B_1^{d_\ell}}$, where $m_i$ denotes the Lebesgue measure on $\mathbb{R}^{d_i}$. The problem obviously splits into a tensor product of $\ell$ one-parameter problems. Consider an example of $\mu_1$ satisfying $\mathcal{H}^{d_1-1}(\text{supp}\,\mu_1)=0$ and
\[
\sup\{\alpha_1:\, \mu_1(B(x,r))\lesssim r^{\alpha_1},\,\forall x\in \mathbb{R}^{d_1},\forall r>0\}= d_1-1,
\]where $\mathcal{H}^{d_1-1}$ denotes the ($d_1-1$)-dimensional Hausdorff measure. Then, one has that $\mu$ satisfies the ball condition as in (2) of Theorem \ref{thm: multipara proj0} for any $\alpha<d-1$. For all $y_1\in \mathbb{R}^{d_1}$ outside the support of $\mu_1$, $\forall p>1$, one obviously has $\|P_{y_1}^{(1)}\mu_1\|_{L^p(S^{d_1-1})}=\infty$. Hence, $\|P_{y}^{(\ell)}\mu\|_{L^p(S^{d_1-1}\times \cdots\times S^{d_\ell-1})}=\infty$, $\forall y\in \mathbb{R}^d$ with $y_i\notin \pi_i({\rm supp}\,\mu)$, $i=1,\cdots, \ell$.

When $\alpha\leq d-1$, we prove a more involved version of Theorem \ref{thm: multipara proj0} in Section \ref{sec: proj} (Theorem \ref{thm: multipara proj}) that incorporates multiparameter orthogonal projections in the integral. Compared to the one-parameter setting, where the energy of the measure is well preserved under the orthogonal projection onto almost every subspace (of an appropriate dimension) according to the Marstrand projection theorem \cite{KM}, the energy of the measures under multiparameter orthogonal projections usually does not behave so well. This is one of the main difficulties one is faced with in extending Theorem \ref{thm: multipara proj0} to Theorem \ref{thm: multipara proj}. This is also one of the reasons why one would need to work with multiparameter analogue of the energy in the proof of both theorems.

As an application, our multiparameter radial projection theorem can be used to study (in-)visibility of sets in the multiparameter setting. More precisely, let $K\subset \mathbb{R}^d=\mathbb{R}^{d_1}\times \cdots \times \mathbb{R}^{d_\ell}$ be a Borel set, where $d_i\geq 2$, $\ell \geq 1$ as before. We say $K$ is \emph{$\ell$-parameter invisible} from $x\in \mathbb{R}^d$ if 
\[
\mathcal{H}^{d-\ell}(P^{(\ell)}_x(K\setminus \{y:x_i=y_i \text{ for some } i=1,\cdots,\ell\}))=0,
\]where $\mathcal{H}^{d-\ell}:=\mathcal{H}^{d_1-1}\big|_{S^{d_1-1}}\times \cdots\times \mathcal{H}^{d_\ell-1}\big|_{S^{d_\ell-1}}$. The set $K$ is said to be \emph{$\ell$-parameter visible} from $x$ if it is not $\ell$-parameter invisible from $x$. Define
\[
{\rm Inv}^{(\ell)}(K)=\{x\in \mathbb{R}^d:\, K \text{ is invisible from }x\}.
\]The basic question is to determine how large the set ${\rm Inv}^{(\ell)}(K)$ can be. The study of invisibility is a classical problem in geometric measure theory. For instance, if $d-1<{\rm dim}(K)\leq d$, Mattila and Orponen \cite{MO, O2} have proved the sharp estimate ${\rm dim}({\rm Inv}^{(1)}(K))\leq 2(d-1)-{\rm dim}(K)$. See Mattila's survey \cite[Section 6]{M2} for a more detailed introduction to this line of research.

A corollary of our Theorem \ref{thm: multipara proj0} is the following sharp estimate. 

\begin{corollary}\label{cor1}
Assume that the Borel set $K\subset \mathbb{R}^d=\mathbb{R}^{d_1}\times \cdots \times \mathbb{R}^{d_\ell}$ satisfies ${\rm dim}(K)>d-1$, where $d_i\geq 2$, $\ell\geq 2$. Then there holds
\begin{equation}\label{eqn: cor1}
{\rm dim}({\rm Inv}^{(\ell)}(K))\leq 2(d-1)-{\rm dim}(K).
\end{equation}
\end{corollary}

As far as the authors are aware of, this seems to be the first estimate of this kind for visibility of sets in the multiparameter setting. In addition, this estimate is sharp. Indeed, let $K_1\subset \mathbb{R}^{d_1}$ be a sharp example for Mattila--Orponen's estimate in the one-parameter setting. Define $K=K_1\times B_1^{d_2+\cdots+d_\ell}$. Then one can easily check that $K$ is a sharp example for (\ref{eqn: cor1}).

Moreover, we have the following slightly stronger corollary, concerning the analogue of the above for measures. Let $\mathbb{R}^d=\mathbb{R}^{d_1}\times \cdots\times \mathbb{R}^{d_\ell}$ be the same as before, $d_i\geq 2$, $\ell\geq 2$.

\begin{corollary}\label{cor2}
Let $\mu\in\mathcal{M}(\mathbb{R}^d)$ satisfy for some $\alpha>d-1$ that
\[
\mu(B^d(x,r))\lesssim r^\alpha,\quad \forall x\in \mathbb{R}^d,\, r>0.
\]Define the set
\[
\begin{split}
S^{(\ell)}(\mu)=&\{x\in \mathbb{R}^d:\, x_i\neq y_i,\, \forall y\in {\rm supp}\mu, \, i=1,\cdots,\ell,\\
&\quad P^{(\ell)}_x\mu \text{ is not absolutely continuous w.r.t. } \mathcal{H}^{d-\ell}\}. 
\end{split}
\]Then there holds 
\[
{\rm dim}(S^{(\ell)}(\mu))\leq 2(d-1)-\alpha.
\]
\end{corollary}

The $\ell=1$ case of this (sharp) estimate is proved by Orponen \cite{O17b}. And a similar argument as above shows that this bound in the $\ell$-parameter setting, $\forall \ell\geq 2$, is sharp as well. Both Corollary \ref{cor1} and \ref{cor2} are directly implied by Theorem \ref{thm: multipara proj0}. For the sake of completeness, a justification is provided in Appendix \ref{appendix: cor}.

\subsection{Structure of the paper and notations}
We provide an outline of the proof in Section \ref{sec: outline} and discuss the pruning process of the Frostman measures in Section \ref{sec: bad}. Section \ref{sec: proj} is devoted to the justification of the multiparameter radial projection theorems (Theorem \ref{thm: multipara proj0} and \ref{thm: multipara proj}) which is independent of the rest of the article. Finally, we prove the multiparameter weighted restriction estimate in Section \ref{sec: dec} which will complete the proof of Theorem \ref{thm: main1}.

Throughout the article, we write $A\lesssim B$ if $A\leq CB$ for some absolute constant $C$; $A\sim B$ if $A\lesssim B$ and $B\lesssim A$; $A\lesssim_\epsilon B$ if $A\leq C_\epsilon B$ for all $\epsilon>0$; $A\lessapprox B$ if $A\leq C_\epsilon R^\epsilon B$ for any $\epsilon>0$, $R>1$.

For any $d\geq 1$, $B_1^d$ denotes the unit ball in $\mathbb{R}^d$ centered at the origin, and $B^{d}(x,r)$ denotes the ball in $\mathbb{R}^{d}$ centered at $x$ of radius $r$. For any $1\leq m\leq n$, $G(n,m)$ denotes the Grassmanian, the space of $m$-dimensional subspaces in $\mathbb{R}^n$. 

For a large parameter $R$, ${\rm RapDec}(R)$ denotes those quantities that are bounded by a huge (absolute) negative power of $R$, i.e. ${\rm RapDec}(R) \leq C_N R^{-N}$ for arbitrarily large $N>0$. Such quantities are negligible in our argument. Similarly, $\text{RapDec}(R_{j_{i_1}},\cdots, R_{j_{i_k}})$ denotes a quantity less than $\text{RapDec}(R_{j_{i_s}})$ for all $1\leq s\leq k$. We say a function is essentially supported in a region if (the appropriate norm of) the tail outside the region is ${\rm RapDec}(R)$ for the underlying parameter $R$.

\begin{acknowledgement} 
XD is supported by NSF DMS-2107729. YO is supported by NSF DMS-2042109. RZ is supported by the NSF grant DMS-1856541, DMS-1926686 and by the Ky Fan and Yu-Fen Fan Endowment Fund at the Institute for Advanced Study. We would like to thank Tuomas Orponen for pointing us to the example of Davies--Fast \cite{DF} and Alex Iosevich for helpful discussions over the course of the project. We are also grateful to the anonymous referee whose suggestions have greatly improved the accuracy and exposition of the article.
\end{acknowledgement}

\section{Proof of Theorem \ref{thm: main1}: outline}\label{sec: outline}
\setcounter{equation}0

Given $E\subset\mathbb{R}^d$ with $\text{dim}(E)>d-\frac{d_{\min}}{2}$, without loss of generality, assume $E$ is contained in the unit ball $B_1^d$. Let $\alpha\in (d-\frac{d_{\min}}{2}, {\rm dim}(E))$, we will construct two subsets $E_1, E_2\subset E$ each of which supports a probability measure $\mu_1$, $\mu_2$ such that
\[
\mu_i(B(x,r))\lesssim r^\alpha,\quad \forall x\in\mathbb{R}^d,\,\forall r>0.
\]The key properties $E_1, E_2$ have are that their images under a fixed collection of orthogonal projections remain well separated and the multiparameter radial projection theorem (Theorem \ref{thm: multipara proj}) applies. The exact construction of the two sets is explained at the end of the section.

Let $x\in E_2$ be any fixed point and let $d^x(y):=(|x_1-y_1|,\dots, |x_\ell-y_\ell|)$ be its induced multiparameter distance map determined by $\vec{d}=(d_1,\ldots, d_\ell)$. Then, the pushforward measure $d^x_\ast(\mu_1)$, defined as
\[
\int_{\mathbb{R}^\ell} \psi(t_1,\cdots,t_\ell)\,d^x_\ast(\mu_1)=\int_{E_1}\psi(|x_1-y_1|,\dots, |x_\ell-y_\ell|)\,d\mu_1(y),
\]is a natural measure that is supported on $\Delta^{\vec{d}}_x(E)$. 

We will construct another complex valued measure $\mu_{1,g}$ that is the \emph{good} part of $\mu_1$ with respect to $\mu_2$, and study how its pushforward under the map $d^x$ differs from $d^x_\ast(\mu_1)$. The main result we will prove is the following.

\begin{theorem}\label{thm: actual}
Let $d=d_1+\cdots +d_\ell$ with $d_i\geq 2$, and $E\subset \mathbb{R}^d$ be a compact set with $\text{dim}(E)>d-\frac{d_{\min}}{2}$. Then for any $\alpha\in (d-\frac{d_{\min}}{2}, {\rm dim}(E))$, there exist $E_1, E_2\subset E$ with ${\text dist}(E_1, E_2)\gtrsim 1$ such that the following are true. 
\begin{enumerate}
\item For $i=1,2$, $E_i$ has positive $\alpha$ dimensional Hausdorff measure and supports a probability measure $\mu_i$ satisfying $\mu_i(B(x,r))\lesssim r^\alpha$, $\forall x\in\mathbb{R}^d,\,\forall r>0$;
\item There exists a complex valued measure $\mu_{1,g}$ and a subset $E_2' \subset E_2$ so that $\mu_2(E_2') \ge 1 - \frac{1}{1000}$ and for each $x \in E_2'$,
\begin{equation}\label{eqn: mainest1}
\| d^x_*(\mu_1) - d^x_*(\mu_{1,g}) \|_{L^1} < \frac{1}{1000}.\end{equation}
\item If in addition, one assumes that
\[
\text{dim}(E)>\alpha>\begin{cases} d-\frac{d_{\min}}{2}+\frac{1}{4},& d_{\min} \text{ is even},\\ d-\frac{d_{\min}}{2}+\frac{1}{4}+\frac{1}{4d_{\min}},& d_{\min} \text{ is odd},\end{cases}
\]then 
\begin{equation}\label{eqn: mainest2}
\int_{E_2}  \| d^x_*(\mu_{1,g}) \|_{L^2}^2 d \mu_2(x) < + \infty. \end{equation}
\end{enumerate} 
\end{theorem}

It is easy to see that Theorem \ref{thm: main1} immediately follows from Theorem \ref{thm: actual}. We briefly sketch the argument below for the sake of completeness. The estimate (\ref{eqn: mainest1}) and (\ref{eqn: mainest2}) above can be viewed as multiparameter versions of \cite[Proposition 2.1 and 2.2]{GIOW}. Since the measures $\mu_1, \mu_2$ are not necessarily tensor products, and our construction of $\mu_{1,g}$ will be more involved compared to its one-parameter analogue, it doesn't seem that one can iterate Proposition 2.1 and 2.2 of \cite{GIOW} directly to obtain (\ref{eqn: mainest1}) and (\ref{eqn: mainest2}).

\begin{proof}[Proof of Theorem \ref{thm: main1} assuming Theorem \ref{thm: actual}]

According to conclusions (2) and (3) of Theorem \ref{thm: actual}, there is a point $x \in E_2$ satisfying (\ref{eqn: mainest1}) and $\| d^x_*(\mu_{1,g}) \|_{L^2} < + \infty$. Since $d^x_*(\mu_1)$ is a probability measure, one has from (\ref{eqn: mainest1}) that $\|d^x_*(\mu_{1,g})\|_{L^1} \ge 1 - 1/1000. $
Note that $d^x_*(\mu_1)$ is supported on $\Delta^{\vec{d}}_x(E)$, hence
$$
\int_{\Delta^{\vec{d}}_x(E)} | d^x_* \mu_{1,g} | = \int |d^x_* (\mu_{1,g})| - \int_{\Delta^{\vec{d}}_x(E)^c} |d^x_*(\mu_{1,g})|
$$
$$
\ge 1 - \frac{1}{1000} - \int |d^x_*(\mu_1) - d^x_*(\mu_{1,g})| \ge 1 - \frac{2}{1000}.
$$
On the other hand,
\begin{equation} \label{yes}
\int_{\Delta^{\vec{d}}_x(E)} | d^x_* \mu_{1,g}| \le | \Delta^{\vec{d}}_x(E)|^{1/2} \left( \int |d^x_* \mu_{1,g}|^2 \right)^{1/2}.
\end{equation}
Since $\int |d^x_* \mu_{1,g}|^2$ is finite, it follows that $|\Delta^{\vec{d}}_x(E)|$ is positive. 
\end{proof}

Now we present the details of the construction of $E_1$ and $E_2$.

Let $E\subset \mathbb{R}^d=\mathbb{R}^{d_1}\times\cdots\times \mathbb{R}^{d_\ell}$ be the fixed compact set as before, satisfying ${\rm dim}(E)>d-\frac{d_{\min}}{2}$. For any $\alpha \in (d-\frac{d_{\min}}{2}, {\rm dim}(E))$, the $\alpha$-dimensional Hausdorff measure of $E$ is positive. Next, for each $i=1,\cdots, \ell$, we arbitrarily fix an $m_i$-dimensional subspace $W_i\subset \mathbb{R}^{d_i}$, where $m_i:=d_i-\frac{d_{\min}}{2}+1$ if $d_{\min}$ is even, and $d_i-\frac{d_{\min}}{2}+\frac{1}{2}$ if $d_{\min}$ is odd. It is easy to see that there always holds $\alpha>d-m_i$. We then construct subsets $E_1, E_2\subset E$ according to the following lemma.

\begin{lemma}\label{lem: E1E2}
Let $E\subset \mathbb{R}^d=\mathbb{R}^{d_1}\times\cdots\times \mathbb{R}^{d_\ell}$ be a compact set with positive $\alpha$-dimensional Hausdorff measure, and $W_i$ an $m_i$-dimensional subspace in $\mathbb{R}^{d_i}$, where $\alpha>d-m_i$, $\forall i=1,\cdots,\ell$. Then there exist subsets $E_1, E_2\subset E$ so that 

(1) both $E_1, E_2$ have positive $\alpha$-dimensional Hausdorff measures;

(2) ${\rm dist}(\pi_{W_i} \circ \pi_i(E_1),\pi_{W_i}\circ \pi_i(E_2))\gtrsim 1$, $\forall i=1,\cdots,\ell$, where $\pi_i: \mathbb{R}^d\rightarrow \mathbb{R}^{d_i}$ and $\pi_{W_i}:\mathbb{R}^{d_i}\rightarrow W_i$ denote the orthogonal projections.
\end{lemma}

\begin{proof}
We first construct a pair of subsets $E_1^1, E_2^1\subset E$ so that they both have 
positive $\alpha$-dimensional Hausdorff measures and their projections onto $W_1$ are separated.
Let $\mu^1$ be the pushforward measure of $\mathcal{H}^\alpha|_{E}$ under the projection $\pi_{W_1}\circ\pi_{1}$. By the assumption that $\alpha>d-m_1$, the support of $\mu^1$ has at least two
distinct points, and so we can take two separated balls $\tilde E_1^1$ and $\tilde E_2^1$ in $W_1$ around these points of positive $\mu^1$ measure. Then their preimages $E_1^1:=(\pi_{W_1}\circ\pi_{1})^{-1}(\tilde E_1^1) \cap E$ and $E_2^1:=(\pi_{W_1}\circ\pi_{1})^{-1}(\tilde E_2^1) \cap E$ satisfy the desired properties.

Next, we construct $E_1^2 \subset E_1^1$ and $E_2^2\subset E_2^1$ so that they both have positive $\alpha$-dimensional Hausdorff measures and their projections onto $W_2$ are separated.
Let $\mu_j^2$ be the pushforward measure of $\mathcal{H}^\alpha|_{E_j^1}$ under the projection $\pi_{W_2}\circ\pi_{2}$, $j=1,2$. By the assumption that $\alpha>d-m_2$, there are one point in the support of $\mu_1^2$ and another point in the support of $\mu_2^2$, and so in $W_2$ around these points we can take two separated balls: $\tilde E_1^2$ of positive $\mu_1^2$ measure and $\tilde E_2^2$ of positive $\mu_2^2$ measure. Then their preimages $E_1^2:=(\pi_{W_2}\circ\pi_{2})^{-1}(\tilde E_1^2) \cap E_1^1$ and $E_2^1:=(\pi_{W_2}\circ\pi_{2})^{-1}(\tilde E_2^2) \cap E_2^1$ satisfy the desired properties.

By repeating the argument in each of the rest of the variables, one eventually obtains subsets $E^{\ell}_1, E^{\ell}_2\subset E$ satisfying the desired conditions. 
\end{proof}

\section{Step 1: construction of $\mu_{1,g}$ and proof of (\ref{eqn: mainest1})}\label{sec: bad}
\setcounter{equation}0

The good measure $\mu_{1,g}$ is going to be defined by eliminating different types of bad multiparameter wave packets from $\mu_1$. .

Let $R_0$ be a large number that will be determined later, and let $R_{\vec{j}} :=(R_{j_{1}},\cdots, R_{j_{\ell}})$ with  $R_{j_i}:= 2^{j_i} R_0$. In each of the $\ell$ components, cover the annulus $R_{j_i-1} \le |\omega_i| \le R_{j_i}$ in $\mathbb{R}^{d_i}$ by rectangular blocks $\tau_i$ with dimensions approximately $R_{j_i}^{1/2} \times \cdots \times R_{j_i}^{1/2}\times R_{j_i}$, with the long direction of each block $\tau_i$ being the radial direction.  Then, choose a partition of unity subordinate to this cover for each $i$ so that

$$ 1 = \psi^i_0 + \sum_{j_i \ge 1, \tau_i} \psi_{j_i, \tau_i}. $$

Let $\delta > 0$ be a small constant to be determined later (more precisely, $\delta$ depends on $\alpha$ and will be determined in the proof of (\ref{eqn: mainest2}) in Section \ref{sec: dec}). For each $(j_i, \tau_i)$, cover the unit ball $B^{d_i}_1$ with tubes $T_i$ of dimensions approximately $R_{j_i}^{-1/2 + \delta}\times \cdots  R_{j_i}^{-1/2 + \delta}\times 2$ with the long axis parallel to the long axis of $\tau_i$. This covering has uniformly bounded overlap. Denote the collection of all these tubes in the $i$-th component as $\mathbb{T}^i_{j_i, \tau_i}$. Let $\eta_{T_i}$ be a smooth cutoff function essentially supported on $T_i$ and with Fourier transform compactly supported in the rectangle of dimensions $R_{j_i}^{1/2}\times\cdots \times R_{j_i}^{1/2}\times 2$ with the short side pointing in the long direction of $T_i$, so that for each choice of $j_i$ and $\tau_i$, $ \sum_{T_i \in \mathbb{T}^i_{j_i, \tau_i}} \eta_{T_i} $ is equal to $1$ on the ball of radius $2$. More precisely, $\eta_{T_i}$ rapidly decays outside the concentric tube of $T_i$ with a constant multiple of the radius.   %Let $\eta_{T_i}$ be a  partition of unity subordinate to this covering, so that for each choice of $j_i$ and $\tau_i$, $ \sum_{T_i \in \mathbb{T}^i_{j_i, \tau_i}} \eta_{T_i} $ is equal to $1$ on the ball of radius $2$ and each $\eta_{T_i}$ is a smooth function. 

For each $T_i \in \mathbb{T}^i_{j_i, \tau_i}$, define an operator acting on the $i$-th variable:
\[
M^i_{T_i} f :=\eta_{T_i}  \mathcal{F}^{-1}_i[\psi_{j_i, \tau_i} \mathcal{F}_i(f)],
\]where $\mathcal{F}_i$ denotes the Fourier transform in the $i$-th variable. Roughly speaking, the operator $M^i_{T_i}$ maps $f$ in the $i$-th variable to the part of it that has Fourier support in $\tau_i$ and physical support in $T_i$.  Define also $M^i_0 f := \mathcal{F}^{-1}_i[\psi^i_0 \mathcal{F}_i(f)]$.  We denote $\mathbb{T}^i_{j_i} = \bigcup_{\tau_i} \mathbb{T}^i_{j_i, \tau_i}$ and $ \mathbb{T}^{i}=\bigcup_{j_i \ge 1} \mathbb{T}^i_{j_i}$. Hence, for any $L^1$ function $f$ supported on the unit ball, up to a small tail, $f$ can be decomposed in terms of these operators.

\begin{lemma}
Let $d=d_1+\cdots +d_\ell$ and $f\in L^1$ be a function supported on the unit ball $B_1^d$. Then
\[
f = \left[M^1_0  + \sum_{T_1 \in \mathbb{T}^1} M^1_{T_1} \right]\cdots\left[M^\ell_0  + \sum_{T_\ell \in \mathbb{T}^\ell} M^\ell_{T_\ell} \right]f+\text{RapDec}(R_0)f_{\text{tail}}\\
\]where the tail satisfies $\|f_{\text{tail}}\|_{L^1}\lesssim \|f\|_{L^1}$. 
\end{lemma}

This decomposition essentially follows from iterating \cite[Lemma 3.4]{GIOW}. Since the tail term becomes more complicated in the multiparameter case, we provide the proof below for the sake of completeness.

\begin{proof}
We prove the case $\ell=2$ here. The proof of the general multiparameter case proceeds similarly. By applying \cite[Lemma 3.4]{GIOW} iteratively (first in the $x_1$ variable, then in $x_2$), one obtains
\[
\begin{split}
f=&\big[M^1_0  + \sum_{T_1 \in \mathbb{T}^1} M^1_{T_1} \big]f+\text{RapDec}(R_0)f_{\text{tail}}^{(1)}\\
=& \left[M^1_0  + \sum_{T_1 \in \mathbb{T}^1} M^1_{T_1} \right]\left[M^2_0  + \sum_{T_2 \in \mathbb{T}^2} M^2_{T_2} \right]f+ \text{RapDec}(R_0)f_{\text{tail}}^{(2)}+\text{RapDec}(R_0)f_{\text{tail}}^{(1)}\,,
\end{split}
\]
where
$$
\|f^{(1)}_{\text{tail}}(\cdot,x_2)\|_{L^1_{x_1}} \leq \|f(\cdot,x_2)\|_{L^1_{x_1}}\,,
$$
and 
$$
\|f^{(2)}_{\text{tail}}(x_1, \cdot)\|_{L^1_{x_2}} \leq \Big\| \big[M^1_0  + \sum_{T_1 \in \mathbb{T}^1} M^1_{T_1} \big]f(x_1,\cdot) \Big\|_{L^1_{x_2}}\,.
$$
Therefore,
$$
\|f^{(1)}_{\text{tail}}\|_{L^1} \leq \|f\|_{L^1}\,,
$$
and 
$$
\|f^{(2)}_{\text{tail}}\|_{L^1} \leq \|f-\text{RapDec}(R_0)f_{\text{tail}}^{(1)}\|_{L^1} \leq \|f\|_{L^1}+\text{RapDec}(R_0)\|f\|_{L^1}\,.
$$
\end{proof}

We now define different types of bad product tubes. Let $c(\alpha)> 0$ be a large constant to be determined later, and for each $i$, let $4T_i$ denote the concentric tube of four times the radius. For any given $1\leq k\leq \ell$ and components $1\leq i_1<\cdots <i_k\leq \ell$, we say a product tube $(T_{i_1},\cdots, T_{i_k}) \in \mathbb{T}^{i_1}_{j_{i_1}}\times \cdots \times \mathbb{T}^{i_k}_{j_{i_k}}$ is \emph{($i_1,\cdots, i_k$)-bad} if
\[
\mu_2 \left(\big(\prod_{s=1}^k 4T_{i_s}\big)\times \big(\prod_{i\neq i_1,\cdots,i_k}B_1^{d_i}\big)\right) \ge \begin{cases} \prod_{s=1}^k R_{j_{i_s}}^{-(\frac{d_{i_s}}{2}-\frac{d_{\min}}{4}) + c(\alpha) \delta}, & d_{\min} \text{ is even}; \\ \prod_{s=1}^k R_{j_{i_s}}^{-(\frac{d_{i_s}}{2}-\frac{d_{\min}}{4}-\frac{1}{4}) + c(\alpha) \delta}, & d_{\min} \text{ is odd}. \end{cases}
\]A product tube $(T_{i_1},\cdots, T_{i_k})$ is \emph{($i_1, \cdots, i_k$)-good} if it is not ($i_1,\cdots, i_k$)-bad.

%Similarly, a tube $T_2 \in \mathbb{T}_{j_2, \tau_2}$ is \emph{($2$)-bad} if
%\[
%\mu_2 (B_2^{d_1}\times 4T_2) \ge \begin{cases} %R_{j_2}^{-(\frac{d_2}{2}-\frac{d_{\min}}{4}) + c(\alpha) \delta}, & %d_{\min} \text{ is even}; \\ %R_{j_2}^{-(\frac{d_2}{2}-\frac{d_{\min}}{4}-\frac{1}{4}) + c(\alpha) %\delta}, & d_{\min} \text{ is odd}. \end{cases}
%\]

%In addition, we need to remove another type of product bad tubes. A product tube $T_1\times T_2\in \mathbb{T}_{j_1, \tau_1} \times \mathbb{T}_{j_2, \tau_2}$ is called \emph{($1,2$)-bad} if
%\[
%\mu_2 (4T_1\times 4T_2) \ge \begin{cases} R_{j_1}^{-(\frac{d_1}{2}-\frac{d_{\min}}{4}) + c(\alpha) \delta}R_{j_2}^{-(\frac{d_2}{2}-\frac{d_{\min}}{4}) + c(\alpha) \delta}, & d_{\min} \text{ is even};\\  R_{j_1}^{-(\frac{d_1}{2}-\frac{d_{\min}}{4}-\frac{1}{4}) + c(\alpha) \delta} R_{j_2}^{-(\frac{d_2}{2}-\frac{d_{\min}}{4}-\frac{1}{4}) + c(\alpha) \delta}, & d_{\min} \text{ is odd}. \end{cases}
%\]

%, similarly for \emph{($2$)-good} and \emph{($1,2$)-good}. Moreover, we say a product tube $T_1\times T_2$ is \emph{excellent} if $T_i$ is ($i$)-good, $i=1,2$, and $T_1\times T_2$ is ($1,2$)-good.

%\begin{remark}
%Note that the different notions of badness defined above are not comparable. Given a ($1,2$)-bad product tube $T_1\times T_2$, it's not necessarily true that either $T_1$ is ($1$)-bad or $T_2$ is ($2$)-bad. On the other hand, even if $T_i$ is ($i$)-bad, for both $i=1,2$, it doesn't seem to be always the case that $T_1\times T_2$ is ($1,2$)-bad.
%\end{remark}

%\subsection{Construction of $\mu_{1,g}$}

We are now ready to construct the complex valued measure $\mu_{1,g}$, by removing all bad product wave packets from $\mu_1$. Define the good part of the measure $\mu_1$ with respect to $\mu_2$ as
\begin{equation} \label{DefG}
\mu_{1,g}=\sum_{k=0}^{\ell} \sum_{1\leq i_1<\cdots<i_k\leq\ell} \bigg(\prod_{i\neq i_1,\cdots,i_k} M^i_0\bigg) \bigg(\sum_{\substack{(T_{i_1},\cdots,T_{i_k})\in \mathbb{T}^{i_1}\times\cdots\times\mathbb{T}^{i_k} \\ (i_1,\cdots,i_k){\textrm -good}}} M^{i_1}_{T_{i_1}}\cdots M^{i_k}_{T_{i_k}}\mu_1\bigg).
\end{equation}
%\[
%\begin{split}
%\mu_{1,g}:=& M^1_0M_0^2 \mu_1 + M_0^1\left(\sum_{T_2 \in \mathbb{T}_2: \,T_2 \textrm{ ($2$)-good}} M^2_{T_2} \right)\mu_1\\
%&+\left(\sum_{T_1 \in \mathbb{T}_1: \,T_1 \textrm{ ($1$)-good}} M^1_{T_1} \right)M_0^2 \mu_1 +\sum_{T_i \in \mathbb{T}_i,\, i=1,2:\, T_1\times T_2 \textrm{ excellent}}M_{T_1}^1M_{T_2}^2\mu_1.
%\end{split}
%\]
Here, the $k=0$ term is $M_0^1M_0^2\cdots M_0^\ell \mu_1$ by convention. We point out that $\mu_{1,g}$ is only a complex valued measure, and is essentially supported in the $R_0^{-1/2+\delta}$-neighborhood of $E_1$ with a rapidly decaying tail away from it (see Lemma \cite[Lemma 5.2]{GIOW} for the proof of the analogue in the one-parameter case, which implies our claim by iteration).

\subsection{Proof of (\ref{eqn: mainest1})}

Fix a point $x\in E_2$, one has by definition that
\begin{equation}\label{eqn: bad}
\begin{split}
&\|d^x_\ast(\mu_1)-d^x_\ast(\mu_{1,g})\|_{L^1}\\
\leq & \text{RapDec}(R_0)+\sum_{k=1}^{\ell} \sum_{1\leq i_1<\cdots<i_k\leq\ell} \sum_{j_{i_1},\cdots, j_{i_k}\geq 1}\\
&\quad  \sum_{\substack{(T_{i_1},\cdots,T_{i_k})\in \mathbb{T}^{i_1}_{j_{i_1}}\times\cdots\times\mathbb{T}^{i_k}_{j_{i_k}} \\ (i_1,\cdots,i_k){\textrm -bad}}}\left\|d_\ast^x\left[\big(\prod_{i\neq i_1,\cdots,i_k} M^i_0\big) M^{i_1}_{T_{i_1}}\cdots M^{i_k}_{T_{i_k}}\mu_1\right]\right\|_{L^1}.
\end{split}
\end{equation}

%by a direct calculation that
%\[
%\|d^x_\ast(\mu_1)-d^x_\ast(\mu_{1,g})\|_{L^1}\leq I+II+ III+ IV,
%\]where
%\[
%I:=\sum_{j_1\geq 1}\Big\|\big(\sum_{T_1\in \mathbb{T}_{j_1}:\, T_1\text{ ($1$)-bad}} d^x_\ast M_{T_1}^1\mu_1\big)  \Big\|_{L^1},\quad  II:=\sum_{j_2\geq 1}\Big\|\big(\sum_{T_2\in \mathbb{T}_{j_2}:\, T_2\text{ ($2$)-bad}} d^x_\ast M_{T_2}^2\mu_1\big)  \Big\|_{L^1},
%\]
%\[
%III:=\sum_{j_1,j_2\geq 1}\Big\|\sum_{T_i\in \mathbb{T}_{j_i}:\, T_i\text{ ($i$)-bad}} d^x_\ast M_{T_1}^1 M_{T_2}^2\mu_1\Big\|_{L^1},
%\]
%\[ 
%IV:=\sum_{j_1,j_2\geq 1}\Big\| \sum_{\substack{T_i\in \mathbb{T}_{j_i}:\, T_i\text{ ($i$)-good},\\ T_1\times T_2 \text{ ($1,2$)-bad}}}d^x_\ast M_{T_1}^1M_{T_2}^2\mu_1 \Big\|_{L^1}.
%\]

In the following, we will first reduce the quantity above to the $\mu_1$ measure of certain bad regions, which are unions of bad product tubes of different types. Such reduction follows very closely the one-parameter case treated in Section 3 of \cite{GIOW}. Many arguments below for this reduction are iterations of their corresponding one-parameter analogues. However, it is sometimes impossible to directly iterate the one-parameter result (see for instance Lemma \ref{lem: hardL1} below), and even if it is possible, oftentimes there are various rapidly decaying tails involved in the reduction, which makes the iteration process quite delicate. Therefore, for the sake of completeness and clarity, we include a sketch of the reduction and provide necessary details for handling new complications in the multiparameter setting. 

To estimate the pushforward measures defined for each product tube, we first need the following lemma.

\begin{lemma}\label{lem: easyL1}
For any given $1\leq k\leq \ell$, any components $1\leq i_1<\cdots <i_k\leq \ell$, and any product tube $(T_{i_1},\cdots, T_{i_k}) \in \mathbb{T}^{i_1}_{j_{i_1}}\times \cdots \times \mathbb{T}^{i_k}_{j_{i_k}}$. Let $f$ be a function supported in the unit ball, then there holds
\begin{equation}\label{eqn: x in}
\|M^{i_1}_{T_{i_1}}\cdots M^{i_k}_{T_{i_k}} f\|_{L^1}\lesssim \sum_{S\subset \{1,\cdots, k\}}\left(\prod_{s\in \{1,\cdots, k\}\setminus S}\text{RapDec}(R_{j_{i_s}})\right)\left\|f\left(\prod_{s\in S} \chi_{2T_{i_s}}\right)\right\|_{L^1}.
\end{equation}
%\[
%\|M_{T_2}^2 f\|_{L^1}\lesssim \|f\|_{L^1( \mathbb{R}^{d_1}\times %2T_2)}+\text{RapDec}(R_{j_2})\|f\|_{L^1},
%\]and
%\[
%\|M_{T_1}^1M_{T_2}^2 f\|_{L^1}\lesssim \|f\|_{L^1(2T_1\times 2T_2)}+\text{RapDec}(R_{\vec{j}})\|f\|_{L^1}.
%\]

Moreover, for any $i=1,\ldots, \ell$, there holds
\begin{equation}\label{eqn: M0}
    \|M^i_0 f\|_{L^1}\lesssim \|f\|_{L^1}.
\end{equation}
\end{lemma}

\begin{proof}
Estimate (\ref{eqn: x in}) follows from the one-parameter analogous result in \cite[Lemma 3.2]{GIOW} and induction on $k$. Without loss of generality, assume that $\{i_1,\cdots, i_k\}=\{1,\cdots,k\}$. First consider the case $k=1$. Denote $x''=(x_2,\cdots,x_\ell)$ and $d''=d_2+\cdots+d_\ell$.
By applying \cite[Lemma 3.2]{GIOW} to the variable $x_1$ one obtains that
\[
\begin{split}
\|M_{T_1}^1 f\|_{L^1}\lesssim &\Big\| \|f(\cdot, x'')\|_{L^1_{x_1}(2T_1)}+\text{RapDec}(R_{j_1})\|f(\cdot,x'')\|_{L^1_{x_1}(\mathbb{R}^{d_1})} \Big\|_{L^1_{x''}}\\
\leq &\|f\|_{L^1(2T_1\times \mathbb{R}^{d''})}+\text{RapDec}(R_{j_1})\|f\|_{L^1}\,.
\end{split}
\]Note that even though we have chosen to work with non compactly supported cutoff functions $\eta_{T_i}$ in the definition of $M^i_{T_i}$, \cite[Lemma 3.2]{GIOW} still holds true.

Now assume that \eqref{eqn: x in} holds in the cases $k=1,2,\cdots,m-1$ and we prove it in the case $k=m$. Applying \eqref{eqn: x in} in the case $k=m-1$ to the function $M^m_{T_m}f$, one has
\begin{align*}
    &\|M^1_{T_1}\cdots M^m_{T_m} f\|_{L^1} \\
    \lesssim &\sum_{S\subset \{1,\cdots, m-1\}}\left(\prod_{s\in \{1,\cdots, m-1\}\setminus S}\text{RapDec}(R_{j_s})\right)\left\|M^m_{T_m}f\left(\prod_{s\in S} \chi_{2T_s}\right)\right\|_{L^1},
\end{align*}
and for each $L^1$ norm on the right hand side there holds
\begin{align*}
    &\left\|M^m_{T_m}f\left(\prod_{s\in S} \chi_{2T_s}\right)\right\|_{L^1} \\
    \lesssim &\left\|f\left(\prod_{s\in S} \chi_{2T_s}\right)\chi_{2T_m}\right\|_{L^1} + \text{RapDec}(R_{j_m}) \left\|f\left(\prod_{s\in S} \chi_{2T_s}\right)\right\|_{L^1}\,,
\end{align*}
from which \eqref{eqn: x in} in the case $k=m$ follows.

For the second estimate (\ref{eqn: M0}), by definition, 
\[M^i_0 f := \mathcal{F}^{-1}_i[\psi^i_0 \mathcal{F}_i(f)]=\mathcal{F}^{-1}_i(\psi^i_0)\ast f.
\]Since $\mathcal{F}^{-1}_i(\psi^i_0)$ is essentially supported in a ball in $\mathbb{R}^{d_i}$ centered at the origin with radius $R_0^{-1}$, and satisfies $\|\mathcal{F}^{-1}_i(\psi^i_0)\|_{L^1}\lesssim 1$, one concludes (\ref{eqn: M0}). 
\end{proof}

A direct consequence of the lemma above is that for each $T_{i_s}\in \mathbb{T}^{i_s}_{j_{i_s}}$, $s=1,\cdots, k$, 
\[
\begin{split}
&\left\| d_\ast^x\left[\big(\prod_{i\neq i_1,\cdots,i_k} M^i_0\big) M^{i_1}_{T_{i_1}}\cdots M^{i_k}_{T_{i_k}}\mu_1\right] \right\|_{L^1}\\
\lesssim& \sum_{S\subset \{1,\cdots, k\}}\left(\prod_{s\in \{1,\cdots, k\}\setminus S}\text{RapDec}(R_{j_{i_s}})\right) \mu_1\left(\big(\prod_{s\in S}2T_{i_s}\big)\times \big(\prod_{i\neq i_s,\,\forall s\in S}B_1^{d_i}\big)\right).
\end{split}
\]
%\[
%\|d^x_\ast(M_{T_1}^1\mu_1)\|_{L^1}\lesssim \mu_1(2T_1\times B_1^{d_2}),\quad \|d^x_\ast(M_{T_2}^2\mu_1)\|_{L^1}\lesssim \mu_1(B_1^{d_1}\times 2T_2),
%\]and
%\[
%\|d^x_\ast(M_{T_1}^1M_{T_2}^2\mu_1)\|_{L^1}\lesssim \mu_1(2T_1\times 2T_2).
%\]
%This provides us with a way to bound the bad terms by the $\mu_1$ measure of the bad tubes. 

There is another piece of information needed for us to bound the total contribution of the bad tubes, that is, it suffices to consider those tubes that are close to the fixed point $x$, at least in one of the variables. In the one-parameter setting studied in \cite{GIOW}, let $d^x(y):=|x-y|$ be the one-parameter distance map, $T\in \mathbb{T}_{j,\tau}$ be a one-parameter tube of radius $R_j^{-1/2+\delta}$ with long direction given by the block $\tau$, and $M_T$ be the associated operator that morally restricts the support of the input function to $T$ and its Fourier support to $\tau$. It is shown in Lemma 3.1 of \cite{GIOW} that $\|d^x_\ast(M_T \mu_1)\|_{L^1}\lesssim \text{RapDec}(R_{j})$ if $x\notin 2T$. In the lemma below, we explore the multiparameter extension of this phenomenon.

\begin{lemma}\label{lem: hardL1}
For any given $1\leq k\leq \ell$, any components $1\leq i_1<\cdots <i_k\leq \ell$, and any product tube $(T_{i_1},\cdots, T_{i_k}) \in \mathbb{T}^{i_1}_{j_{i_1}}\times \cdots \times \mathbb{T}^{i_k}_{j_{i_k}}$. Let $x\in E_2$. Suppose that for $A\subset \{1,\cdots, k\}$, one has $x_{i_a}\notin 2T_{i_a}$, $\forall a\in A$. Then,
\[
\begin{split}
&\left\| d_\ast^x\left[\big(\prod_{i\neq i_1,\cdots,i_k} M_0^i\big)M_{T_{i_1}}^{i_1}\cdots M_{T_{i_k}}^{i_k} \mu_1 \right] \right\|_{L^1}\\
%\lesssim & \left(\prod_{a\in A}\text{RapDec}(R_{j_{i_a}})\right) \left\| \big( \prod_{s\in \{1,\cdots,k\}\setminus A}M^{i_s}_{T_{i_s}}\big)\mu_1\right\|_{L^1}\\
\lesssim& \left(\prod_{a\in A}\text{RapDec}(R_{j_{i_a}})\right)\cdot\\
&\quad\sum_{S\subset \{1,\cdots, k\}\setminus A} \left(\prod_{s\in \{1,\cdots, k\}\setminus (S\cup A)}\text{RapDec}(R_{j_{i_s}})\right)\mu_1\left(\big(\prod_{s\in S}2T_{i_s} \big) \times \big( \prod_{i\neq i_s,\,\forall s\in S} B_1^{d_i}\big) \right).
\end{split}
\]

%Let $T_i\in\mathbb{T}_{j_i,\tau_i}$, $i=1,2$, and $x\in E_2$. Then one has the following conclusions.
%\begin{enumerate}
%\item For $i=1,2$, if $x_i\notin 2T_i$, then $\|d^x_\ast(M_{T_i}^i\mu_1)\|_{L^1}\lesssim \text{RapDec}(R_{j_i})$.\\
%\item $\|d^x_\ast(M_{T_1}^1M_{T_2}^2\mu_1)\|_{L^1}\lesssim \text{RapDec}(R_{\vec{j}})$ if $x_1\notin 2T_1$ and $x_2\notin  2T_2$.\\
%\item $\|d^x_\ast(M_{T_1}^1M_{T_2}^2\mu_1)\|_{L^1}\lesssim \text{RapDec}(R_{j_1})\mu_1(B_1^{d_1}\times 2T_2)+\text{RapDec}(R_{\vec{j}})$ if $x_1 \notin 2T_1$.\\
%\item $\|d^x_\ast(M_{T_1}^1M_{T_2}^2\mu_1)\|_{L^1}\lesssim \text{RapDec}(R_{j_2})\mu_1(2T_1\times B_1^{d_2})+\text{RapDec}(R_{\vec{j}})$ if $x_2 \notin 2T_2$.
%\end{enumerate}
\end{lemma}

\begin{proof}
By a standard limiting argument, it suffices to study the case that $d\mu_1=\mu_1(y)dy$. In this case, according to Lemma \ref{lem: easyL1}, it suffices to prove that
\begin{equation}\label{eqn: special f}
\begin{split}
&\left\| d_\ast^x\left[\big(\prod_{i\neq i_1,\cdots, i_k} M_0^i \big)M_{T_{i_1}}^{i_1}\cdots M_{T_{i_k}}^{i_k} \mu_1 \right] \right\|_{L^1}\\
\lesssim & \left(\prod_{a\in A}\text{RapDec}(R_{j_{i_a}})\right) \left\| \big(\prod_{i\neq i_1,\cdots,i_k} M_0^i\big)\big( \prod_{s\in \{1,\cdots,k\}\setminus A}M^{i_s}_{T_{i_s}}\big)\mu_1\right\|_{L^1}.
\end{split}
\end{equation}We will in fact prove a more general estimate: for any $f\in L^1$ with compact support and such that ${\rm dist}(x_{i},\pi_{i}({\rm supp}f))\gtrsim 1$, $\forall i=1,\cdots,\ell$, there holds
\begin{equation}\label{eqn: general f}
\left\| d^x_\ast \left[\prod_{a\in A} M_{T_{i_a}}^{i_a} f \right] \right\|_{L^1} \lesssim \left(\prod_{a\in A}\text{RapDec}(R_{j_{i_a}})\right) \|f\|_{L^1}.
\end{equation}It is easy to see that (\ref{eqn: special f}) will follow from (\ref{eqn: general f}) by taking 
\[
f=(\prod_{i\neq i_1,\cdots,i_k} M_0^i)( \prod_{s\in \{1,\cdots,k\}\setminus A}M^{i_s}_{T_{i_s}})\mu_1.\]

For the sake of brevity, we assume that $\ell=2$, $\{i_1,\cdots,i_k\}=\{1, 2\}$ and $A=\{1,2\}$ (i.e. $x_i\notin 2T_i$, $i=1,2$). Then the desired estimate becomes
\begin{equation}\label{eqn: x out}
\|d^x_\ast(M_{T_1}^1 M_{T_2}^2 f)\|_{L^1}\lesssim \text{RapDec}(R_{j_1}, R_{j_2})\|f\|_{L^1}.
\end{equation}The general case can be justified in the exact same way. 

Estimate (\ref{eqn: x out}) will follow from iterating its one-parameter counterpart (\cite[Lemma 3.1]{GIOW}). More precisely, the proof of \cite[Lemma 3.1]{GIOW} implies the following result in the one-parameter setting: for any $g\in L^1$ with compact support such that ${\rm dist}(x,{\rm supp}g)\gtrsim 1$, if $x\notin 2T_j$ (where $T_j\in\mathbb{T}_j$), then 
\[
\|d^x_\ast(M_{T_j}g)\|_{L^1}\lesssim {\rm RapDec}(R_j)\|g\|_{L^1}.
\]

%Since the pushforward measure is essentially supported on $t_1, t_2\sim 1$, it suffices to show for each fixed pair of $t_1, t_2$, that there holds 
%\begin{equation}\label{eqn: hard1}
 %|d^x_\ast(M_{T_1}^1 M_{T_2}^2 f)(t_1,t_2)|\lesssim \text{RapDec}(R_{j_1, }R_{j_2})\|f\|_{L^1}.
%\end{equation}

Back in the multiparameter setting, by definition, one has
\[
\begin{split}
d^x_\ast(M_{T_1}^1 M_{T_2}^2 f)(t_1,t_2)=&M_{T_1}^1 M_{T_2}^2 f \ast^{(1)}\lambda_{t_1}\ast^{(2)}\lambda_{t_2}(x_1,x_2)\\
=& \int_{S^{d_1-1}(x_1,t_1)}\int_{S^{d_2-1}(x_2,t_2)} M_{T_1}^1M_{T_2}^2 f(y_1,y_2)\, d\lambda_{t_2}(y_2) d\lambda_{t_1}(y_1),
\end{split}
\]where $d\lambda_{t_i}$ denotes the surface measure on $S^{d_i-1}(x_i,t_i)$ and $*^{(i)}$
stands for the convolution in the $i$-th variable. This suggests that
\[
d^x_\ast(M_{T_1}^1 M_{T_2}^2 f)(t_1,t_2)=d^{x_2}_\ast[M_{T_2}^2(d^{x_1}_\ast(M_{T_1}^1f))(t_1)](t_2).
\]

By Fubini, the desired bound (\ref{eqn: x out}) then follows from applying the one-parameter estimate first to $g=d_\ast^{x_1}(M_{T_1}^1 f)(t_1)$ in the $x_2$ variable and then to $g=f$ in the $x_1$ variable.
\end{proof}

We are now ready to go back and bound the expression (\ref{eqn: bad}). Since there are finitely many choices of $1\leq k\leq\ell$ and $1\leq i_1<\cdots <i_k\leq \ell$, it suffices to show that the term corresponding to each fixed $k$ and $\{i_1, \cdots,i_k\}$ is bounded as desired. 

Fix a choice of $1\leq k\leq \ell$ and components $1\leq i_1<\cdots <i_k\leq \ell$. Because of symmetry, we assume that $\{i_1,\cdots, i_k\}=\{1,\cdots, k\}$ without loss of generality. Our goal is thus to show that there exists a subset $E'_2\subset E_2$ so that $\mu_2(E'_2)\geq 1-\frac{1}{1000 C_\ell}$ and for each $x\in E'_2$ there holds
\begin{equation}\label{eqn: bad 1}
\sum_{j_1,\cdots,j_k\geq 1} \sum_{\substack{(T_{1},\cdots,T_{k})\in \mathbb{T}^{1}_{j_{1}}\times\cdots\times\mathbb{T}^{k}_{j_{k}} \\ (1,\cdots,k){\textrm -bad}}}\left\|d_\ast^x\left[\big(\prod_{i=1}^k M^{i}_{T_{i}} \big)\big(\prod_{i=k+1}^\ell M^i_0\big) \mu_1\right]\right\|_{L^1}\leq\frac{1}{1000 C_\ell}.
\end{equation}for sufficiently large $R_0$, where $C_\ell$ is some large constant depending on $\ell$.

For any point $x\in E_2$ and $j_{i_1},\cdots,j_{i_k}\geq 1$, the inner sum on the LHS of the above can be bounded as
\[
\begin{split}
    &\sum_{\substack{(T_{1},\cdots,T_{k})\in \mathbb{T}^{1}_{j_{1}}\times\cdots\times\mathbb{T}^{k}_{j_{k}} \\ (1,\cdots,k){\textrm -bad}}}\left\|d_\ast^x\left[\big(\prod_{i=1}^k M^{i}_{T_{i}} \big)\big(\prod_{i=k+1}^\ell M^i_0\big) \mu_1\right]\right\|_{L^1}\\
    \leq &\sum_{A\subset \{1,\cdots, k\}} \sum_{\substack{(T_{1},\cdots,T_{k})\in \mathbb{T}^{1}_{j_{1}}\times\cdots\times\mathbb{T}^{k}_{j_{k}} \\ (1,\cdots,k){\textrm -bad},\\ x_{a}\notin 2T_{a},\, \forall a\in A,\\ x_{a}\in 2T_{a},\, \forall a\in \{1,\cdots, k\}\setminus A}}\left\|d_\ast^x\left[\big(\prod_{i=1}^k M^{i}_{T_{i}} \big)\big(\prod_{i=k+1}^\ell M^i_0\big)\mu_1\right]\right\|_{L^1}\\
    \leq &\sum_{A\subset \{1,\cdots, k\}} \sum_{S\subset \{1,\cdots, k\}\setminus A} \left(\prod_{s\in \{1,\cdots, k\}\setminus S}\text{RapDec}(R_{j_{s}})\right)\cdot\\
    &\qquad \sum_{\substack{(T_{1},\cdots,T_{k})\in \mathbb{T}^{1}_{j_{1}}\times\cdots\times\mathbb{T}^{k}_{j_{k}} \\ (1,\cdots,k){\textrm -bad},\\ x_{s}\in 2T_{s},\, \forall s\in S}} \mu_1\left(\big(\prod_{s\in S}2T_{s} \big) \times \big( \prod_{i\neq s,\,\forall s\in S} B_1^{d_i}\big) \right),
\end{split}
\]where we have applied Lemma \ref{lem: hardL1} in the second inequality.

In the following, we will fix a choice of $A$ and $S$, which is fine since there are only finitely many possibilities. Without loss of generality, assume that $S=\{1,\cdots,m\}$, $m\leq k$. For any given tubes $T_i\in \mathbb{T}^i_{j_i}$, $i=m+1,\cdots,k$, define the set
\[
\begin{split}
\text{Bad}_{j_1,\cdots,j_k}^{T_{m+1},\cdots, T_k}:=&\big\{(y,x)\in E_1\times E_2:\, \exists (T_1,\cdots, T_m) \in \mathbb{T}^1_{j_1}\times \cdots\times \mathbb{T}^m_{j_m} \text{ s.t. }\\
&\quad T_1\times \cdots \times T_k \text{ is ($1,\cdots, k$)-bad},\, y_i, x_i\in 2T_i,\,\forall i=1,\cdots, m  \big\}.
\end{split}
\]Then, observing that the sets $E_1, E_2$ are separated by distance $\gtrsim 1$ in every variable (as a consequence of Lemma \ref{lem: E1E2}) hence the tubes $T_i$ being measured have only finite overlap (up to constant $R_{j_i}^{\delta d_i}$), one thus has  
\begin{equation}\label{eqn: bad 2}
\begin{split}
&\sum_{j_1,\cdots,j_k\geq 1} \left(\prod_{i=m+1}^k \text{RapDec}(R_{j_i}) \right) \sum_{\substack{(T_{1},\cdots,T_{k})\in \mathbb{T}^{1}_{j_{1}}\times\cdots\times\mathbb{T}^{k}_{j_{k}} \\ (1,\cdots,k){\textrm -bad},\\ x_{i}\in 2T_{i},\, 1\leq i\leq m}} \mu_1\left(\big(\prod_{i=1}^m 2T_{i} \big) \times \big( \prod_{i=m+1}^\ell B_1^{d_i}\big) \right)\\
\lesssim & \sum_{j_1,\cdots,j_k\geq 1} \left(\prod_{i=m+1}^k \text{RapDec}(R_{j_i}) \right)\left(\prod_{i=1}^m R_{j_i}^{\delta d_i} \right)\cdot\\
&\quad \sum_{T_{i}\in \mathbb{T}^{i}_{j_{i}},\, m+1\leq i\leq k} \mu_1\left(\big( \bigcup_{\substack{(T_{1},\cdots,T_{m})\in \mathbb{T}^{1}_{j_{1}}\times\cdots\times\mathbb{T}^{m}_{j_{m}}: \\ T_1\times\cdots\times T_k \,(1,\cdots,k){\textrm -bad},\\ x_{i}\in 2T_{i},\, 1\leq i\leq m}}  \prod_{i=1}^m 2T_i\big) \times \big( \prod_{i=m+1}^\ell B_1^{d_i}\big) \right)\\
=&\sum_{j_1,\cdots,j_k\geq 1} \left(\prod_{i=m+1}^k \text{RapDec}(R_{j_i}) \right)\left(\prod_{i=1}^m R_{j_i}^{\delta d_i}\right) \sum_{T_{i}\in \mathbb{T}^{i}_{j_{i}},\, m+1\leq i\leq k} \mu_1(\text{Bad}_{j_1,\cdots,j_k}^{T_{m+1},\cdots, T_k}(x)).
\end{split}
\end{equation}Here, we have used the notation $\text{Bad}_{j_1,\cdots,j_k}^{T_{m+1},\cdots, T_k}(x)$ to denote the slice of the set $\text{Bad}_{j_1,\cdots,j_k}^{T_{m+1},\cdots, T_k}$ at $x\in E_2$.

The key estimate of the subsection is the following.
\begin{lemma}\label{lem: step1main}
Let $\alpha>d-\frac{d_{\min}}{2}$, there exists sufficiently large $c(\alpha)>0$ such that for all $j_1, \cdots, j_k\geq 1$ and all $(T_{m+1},\cdots, T_k)\in \mathbb{T}^{m+1}_{j_{m+1}}\times\cdots\times \mathbb{T}^k_{j_k}$,
%\begin{equation}\label{eqn: Bad1Bad2}
%\mu_1\times \mu_2(\text{Bad}_{j_1})\lesssim R_{j_1}^{-4\delta},\, \mu_1\times \mu_2(\text{Bad}_{j_2})\lesssim R_{j_2}^{-4\delta},
%\end{equation}and
%\begin{equation}\label{eqn: Bad12}
%\mu_1\times \mu_2(\text{Bad}'_{j_1,j_2})\lesssim R_{j_1}^{-2\delta}R_{j_2}^{-2\delta}.\end{equation}
\begin{equation}\label{eqn: bad mu}
\mu_1\times \mu_2(\text{Bad}_{j_1,\cdots,j_k}^{T_{m+1},\cdots, T_k})\leq C(R_{j_{m+1}},\cdots, R_{j_k}) \prod_{i=1}^m R_{j_i}^{-2\delta d_i}.
\end{equation}
\end{lemma}

%\begin{lemma}\label{lem: step1main2}
%Let $\alpha>d-\frac{d_{\min}}{2}$, there exists sufficiently large $c(\alpha)>0$ such that for all $j_1, j_2\geq 1$,
%\[
%\mu_1\times \mu_2(\text{Bad}_{j_1,j_2})\lesssim R_{j_1}^{-2\delta}R_{j_2}^{-2\delta}.
%\]
%\end{lemma}

\begin{proof}[Proof of estimate (\ref{eqn: mainest1}) using Lemma \ref{lem: step1main}]
According to the reduction explained above, it suffices to show that there exists $E_2'\subset E_2$ with $\mu_2(E_2')\geq 1-\frac{1}{1000 C_\ell}$ such that $\forall x\in E_2'$, there holds that 
\begin{align*}
   & \sum_{j_1,\cdots,j_k\geq 1} \left(\prod_{i=m+1}^k \text{RapDec}(R_{j_i}) \right)\left(\prod_{i=1}^m R_{j_i}^{\delta d_i} \right) \sum_{T_{i}\in \mathbb{T}^{i}_{j_{i}},\, m+1\leq i\leq k} \mu_1(\text{Bad}_{j_1,\cdots,j_k}^{T_{m+1},\cdots, T_k}(x))\\
   \leq &\frac{1}{1000 C_\ell}\,,
\end{align*}
if $R_0$ is chosen sufficiently large.

%$I, II, III, IV$ are bounded by $\frac{1}{1000}$ if ${R}^1_0, R_0^2$ are chosen sufficiently large. As the proof for the four terms are identical, we only sketch the argument for term $III$ and omit the rest of the details. 

Note that
\[
\mu_1 \times \mu_2 (\text{Bad}_{j_1,\cdots,j_k}^{T_{m+1},\cdots, T_k}) = \int \mu_1(\text{Bad}_{j_1,\cdots,j_k}^{T_{m+1},\cdots, T_k}(x)) d\mu_2(x). 
\]
Then, by (\ref{eqn: bad mu}), one is able to choose a subset $B_{j_1,\cdots,j_k}^{T_{m+1},\cdots,T_k} \subset E_2$ so that 
\[
\mu_2(B_{j_1,\cdots,j_k}^{T_{m+1},\cdots,T_k}) \le \big(\prod_{i=1}^m R_{j_i}^{- (1/2) \delta d_i}\big) \big(\prod_{i=m+1}^k R_{j_i}^{-N_i}\big)
\]and for all $x \in E_2 \setminus B_{j_1,\cdots,j_k}^{T_{m+1},\cdots,T_k}$, 
\[
\mu_1(\text{Bad}_{j_1,\cdots,j_k}^{T_{m+1},\cdots, T_k}(x)) \leq C(R_{j_{m+1}},\cdots,R_{j_k})\big(\prod_{i=m+1}^k R_{j_i}^{N_i}\big) \big(\prod_{i=1}^m R_{j_i}^{- (3/2) \delta d_i}\big). 
\]Here, $N_{m+1},\cdots, N_k$ are sufficiently large numbers that are chosen so that
\[
\sum_{T_{i}\in \mathbb{T}^{i}_{j_{i}},\, m+1\leq i\leq k} \mu_2(B_{j_1,\cdots,j_k}^{T_{m+1},\cdots,T_k}) \le \prod_{i=1}^k R_{j_i}^{- (1/2) \delta d_i}.
\]

 Define
 \[
 E_2' = E_2 \setminus \left(\bigcup_{j_1,\cdots, j_k \ge 1}\bigcup_{T_{i}\in \mathbb{T}^{i}_{j_{i}},\, m+1\leq i\leq k} B_{j_1,\cdots,j_k}^{T_{m+1},\cdots,T_k}\right).
 \]By taking $R_0$ sufficiently large, one has $\mu_2(E_2') \ge 1 - \frac{1}{1000 C_\ell}$ as desired.  Moreover, for each $x \in E_2'$, there holds
\[
\begin{split}
&\sum_{j_1,\cdots,j_k\geq 1} \left(\prod_{i=m+1}^k \text{RapDec}(R_{j_i}) \right)\left(\prod_{i=1}^m R_{j_i}^{\delta d_i} \right) \sum_{T_{i}\in \mathbb{T}^{i}_{j_{i}},\, m+1\leq i\leq k} \mu_1(\text{Bad}_{j_1,\cdots,j_k}^{T_{m+1},\cdots, T_k}(x))\\
\lesssim &\sum_{j_1,\cdots, j_k \ge 1} \prod_{i=1}^m R_{j_i}^{\delta d_i-(3/2)\delta d_i} \cdot \prod_{i=m+1}^k \text{RapDec}(R_{j_i}) \lesssim R_0^{- (1/2) \delta d_i}. 
\end{split}
\]Hence the desired bound follows by choosing $R_0$ sufficiently large.

\end{proof}

The crucial ingredient in the proof of Lemma \ref{lem: step1main} is a new multiparameter radial projection theorem (Theorem \ref{thm: multipara proj}), which extends the one-parameter version proved by Orponen \cite{O17b}. Theorem \ref{thm: multipara proj} is in fact more general than Orponen's theorem, as it includes the case that $\alpha$ is small as well. The statement and the proof of Theorem \ref{thm: multipara proj}, as well as the proof of Lemma \ref{lem: step1main}, will be given in the next section.

\section{Proof of the multiparameter radial projection theorem}\label{sec: proj}
\setcounter{equation}0

We prove the new multiparameter radial projection theorem in this section, which may be of independent interest. We first prove the special case of the theorem, Theorem \ref{thm: multipara proj0}, where $\alpha$ is assumed to be very large. It is then extended to the general case later. The proof of Lemma \ref{lem: step1main}, which implies estimate (\ref{eqn: mainest1}), is presented at the end of the section.

Below are several ingredients used in the proof of Theorem \ref{thm: multipara proj0}. 

Denote $e=(e_1,\cdots,e_\ell)\in S^{d_1-1}\times\cdots\times S^{d_{\ell}-1}$, and consider the orthogonal projection 
$$\pi_e=\pi_{e_1}\times\cdots\times\pi_{e_\ell}: \mathbb{R}^d= \mathbb{R}^{d_1}\times\cdots\times\mathbb{R}^{d_\ell} \to e_1^\perp\times\cdots\times e_\ell^\perp,
$$
where $e_i^{\perp}\in G(d_i,d_{i}-1)$ is the orthogonal complement of the vector $e_i$ in $\mathbb{R}^{d_i}$ and $\pi_{e_i}: \mathbb{R}^{d_i} \to e_i^\perp$ is the corresponding orthogonal projection. Given $\mu\in\mathcal{M}(\mathbb{R}^d)$ and $y\in\mathbb{R}^d$ with $y_i\notin \pi_i({\rm supp}\,\mu)$, define
$$
\mu_y(x)=C_{d_1,\cdots,d_\ell} \prod_{i=1}^\ell |x_i-y_i|^{1-d_i} d\mu(x),
$$
where $C_{d_1,\cdots,d_\ell}$ is chosen to make Lemma \ref{projlem1} below true.

\begin{lemma}\label{projlem1}
Let $\mu\in C_c(\mathbb{R}^d)$ and $\nu\in\mathcal{M}(\mathbb{R}^d)$ with $\pi_i({\rm supp}\,\mu) \cap \pi_i({\rm supp}\,\nu) = \varnothing$, $\forall i=1,\ldots, \ell$. Then, for $p\in(0,\infty)$,
$$
\int \|P^{(\ell)}_y\mu_y\|^p_{L^p(S^{d_1-1}\times\cdots\times S^{d_{\ell}-1})}\,d\nu(y) = \int_{S^{d_1-1}\times\cdots\times S^{d_{\ell}-1}} \|\pi_e\mu\|_{L^p(\pi_e\nu)}^p d\mathcal{H}^{d-\ell}(e)\,,
$$where $\mathcal{H}^{d-\ell}:=\mathcal{H}^{d_1-1}\big|_{S^{d_1-1}}\times \cdots\times \mathcal{H}^{d_\ell-1}\big|_{S^{d_\ell-1}}$. 
\end{lemma}

Lemma \ref{projlem1} follows from exactly the same argument as its known one-parameter analogue, see \cite[Lemma 3.1]{O17b}. We will also use the following generalized formula for integration  in polar coordinates, see \cite[(24.2)]{MattilaBook}:

\begin{lemma}\label{projlem2}
For any non-negative Borel function $f$ on $\mathbb{R}^n$,
\begin{equation}\label{eqn: polar}
\int_{G(n,k)}\int_{V^\perp} |x|^a f(x)\,d\mathcal{H}^{n-k}(x) d \gamma_{n,k}(V)=C_{n,k}\int_{\mathbb{R}^n}|y|^{a-k}f(y)\,dy,
\end{equation}
where $\gamma_{n,k}$ is the Haar measure on $G(n,k)$.
\end{lemma}

\begin{lemma}\label{projlem3}
Let $0<\alpha\leq d$. Let $\mu\in\mathcal{M}(\mathbb{R}^d)$ with $\mu(B(x,r))\leq  C_\alpha(\mu) r^\alpha$, $\forall x\in \mathbb{R}^d$, $\forall r>0$. Then
$$
\int_{\mathbb{R}^d}\int_{\mathbb{R}^d} \frac{d\mu(x)d\mu(y)}{\prod_{i=1}^\ell |x_i-y_i|^{t_i}} \approx \int_{\mathbb{R}^d} \frac{|\widehat\mu(\xi)|^2}{\prod_{i=1}^\ell |\xi_i|^{d_i-t_i}}\,d\xi \lesssim C_\alpha(\mu)^2\,,
$$
whenever $0<t_i<\alpha-d+d_i, \forall i=1,\cdots,\ell$. Here the implicit constant in $\lesssim$ can depend on the diameter of the support of $\mu$.
\end{lemma}

\begin{proof}
Let the diameter of the support of $\mu$ be $L$. Note that for $(r_1,\cdots,r_\ell)$ with $r_j=\min\{r_i:i=1,\cdots,\ell\}$ and $r_i\lesssim L, \forall i$,
\begin{align*}
&\mu\left( B^{d_1}(x_1,r_1)\times\cdots\times B^{d_\ell}(x_\ell,r_\ell)\right) \\
\lesssim &  C_\alpha(\mu) r_j^\alpha \prod_{i\neq j} \left(\frac{r_i}{r_j}\right)^{d_i} = C_\alpha(\mu) r_j^{\alpha-d+d_j} \prod_{i\neq j} r_i^{d_i} \lesssim C_\alpha(\mu) \prod_{i=1}^\ell r_i^{\alpha-d+d_i}\,.
\end{align*}
Then, by decomposing into regions $2^{-j_i-1}L<|x_i-y_i|\leq 2^{-j_i}L$, we get that
$$
\int_{\mathbb{R}^d}\int_{\mathbb{R}^d} \frac{d\mu(x)d\mu(y)}{\prod_{i=1}^\ell |x_i-y_i|^{t_i}} \lesssim C_\alpha(\mu)^2 \sum_{j_1,\cdots,j_\ell=0}^\infty\, \prod_{i=1}^\ell \frac{2^{-j_i(\alpha-d+d_i)}}{2^{-j_it_i}}\lesssim C_\alpha(\mu)^2\,,
$$
provided that $t_i<\alpha-d+d_i, \forall i=1,\cdots,\ell$, as desired.
\end{proof}

\begin{lemma}\label{projlem4}
Let $\nu\in\mathcal{M}(\mathbb{R}^d)$. Let $\sigma\in\mathcal{M}(S^{d_1-1}\times\cdots\times S^{d_\ell-1})$ with
$$
\sigma\left((B^{d_1}(x_1,r_1)\times\cdots\times B^{d_\ell}(x_\ell,r_\ell))\cap(S^{d_1-1}\times\cdots\times S^{d_\ell-1})\right) \lesssim r_1^{\beta_1}\cdots r_\ell^{\beta_\ell}\,.
$$
Then
\begin{align*}
&\int_{S^{d_1-1}\times\cdots\times S^{d_\ell-1}}\int_{e_1^\perp\times\cdots\times e_\ell^\perp}\int_{e_1^\perp\times\cdots\times e_\ell^\perp}\frac{d\pi_e\nu(z)d\pi_e\nu(w)}{\prod_{i=1}^\ell |z_i-w_i|^{t_i}}\,d\sigma(e)\\
\lesssim &\int_{\mathbb{R}^d}\int_{\mathbb{R}^d} \frac{d\nu(x)d\nu(y)}{\prod_{i=1}^\ell |x_i-y_i|^{t_i}}\,,
\end{align*}
whenever $t_i<\beta_i, \forall i=1,\cdots,\ell$.
\end{lemma}

\begin{proof}
We first notice that 
$$
\int_{S^{d_1-1}\times\cdots\times S^{d_\ell-1}} \frac{d\sigma(e)}{\prod_{i=1}^\ell |\pi_{e_i}(x_i)|^{t_i}}\lesssim \frac{1}{\prod_{i=1}^\ell |x_i|^{t_i}}.
$$
Indeed, for fixed $x_i\in \mathbb{R}^{d_i}$, the subset $\{e_i\in S^{d_i-1}: |\pi_{e_i}(x_i)|\leq 2^{-j_i}|x_i|\}$ is contained in a spherical cap of radius $\sim 2^{-j_i}$, and hence, by decomposing into regions $2^{-j_i-1}|x_i|<|\pi_{e_i}(x_i)|\leq 2^{-j_i}|x_i|$, we get that 
$$
\int_{S^{d_1-1}\times\cdots\times S^{d_\ell-1}} \frac{d\sigma(e)}{\prod_{i=1}^\ell |\pi_{e_i}(x_i)|^{t_i}} \lesssim \sum_{j_1,\cdots,j_\ell=0}^{\infty}\, \prod_{i=1}^\ell \frac{2^{-j_i\beta_i}}{(2^{-j_i}|x_i|)^{t_i}}\lesssim \frac{1}{\prod_{i=1}^\ell |x_i|^{t_i}}\,,
$$
provided that $t_i<\beta_i, \forall i=1,\cdots,\ell$. 

Therefore,
\begin{align*}
&\int_{S^{d_1-1}\times\cdots\times S^{d_\ell-1}}\int_{e_1^\perp\times\cdots\times e_\ell^\perp}\int_{e_1^\perp\times\cdots\times e_\ell^\perp}\frac{d\pi_e\nu(z)d\pi_e\nu(w)}{\prod_{i=1}^\ell |z_i-w_i|^{t_i}}\,d\sigma(e)\\
= &  \int_{S^{d_1-1}\times\cdots\times S^{d_\ell-1}} \int_{\mathbb{R}^d}\int_{\mathbb{R}^d} \frac{d\nu(x)d\nu(y)}{\prod_{i=1}^\ell|\pi_{e_i}(x_i-y_i)|^{t_i}} d\sigma(e) \\ \lesssim & \int_{\mathbb{R}^d}\int_{\mathbb{R}^d} \frac{d\nu(x)d\nu(y)}{\prod_{i=1}^\ell |x_i-y_i|^{t_i}}\,.
\end{align*}

\end{proof}

\begin{proof}[Proof of Theorem \ref{thm: multipara proj0}]
Fix $\alpha>d-1$ and $\beta>2(d-1)-\alpha$, we choose $\delta>0$ such that $\alpha-(d-1)>\delta>(d-1)-\beta$. We'll prove that
$$
\int \|P^{(\ell)}_y\mu\|^p_{L^p(S^{d_1-1}\times\cdots\times S^{d_{\ell}-1})}\,d\nu(y) \lesssim C_\alpha(\mu)^p C_\beta(\nu),
$$
whenever 
\begin{equation} \label{range of p}
1<p<\min\left\{\frac{2(d_i-1)}{2(d_i-1)-\delta}\,, \frac{\beta-d+d_i}{d_i-1-\delta}: i=1,\cdots, \ell \right\}.
\end{equation}
Note that the choice of $\delta$ guarantees that the right hand side of \eqref{range of p} lies in $(1,2)$, so the range of $p$ is nonempty.

Let $\{\psi_n: n\in\mathbb{N}\}$ be a standard approximation of identity on $\mathbb{R}^d$. Given $\mu \in \mathcal{M}(\mathbb{R}^d)$, $\mu_n:=\mu * \psi_n \to \mu$ weakly and so $P^{(\ell)}_y\mu_n \to P^{(\ell)}_y\mu$ weakly for $y\in {\rm supp}\,\nu$, due to the assumption that the supports of $\mu$ and $\nu$ have separated projections. Then, by Fatou's Lemma,
$$
\int \|P^{(\ell)}_y\mu\|^p_{L^p(S^{d_1-1}\times\cdots\times S^{d_{\ell}-1})}\,d\nu(y) \leq\liminf_{n\to\infty} \int \|P^{(\ell)}_y\mu_n\|^p_{L^p(S^{d_1-1}\times\cdots\times S^{d_{\ell}-1})}\,d\nu(y)\,.
$$
Note that $C_\alpha(\mu_n)\leq C_\alpha(\mu), \forall n\in\mathbb{N}$. Therefore, to prove Theorem \ref{thm: multipara proj0}, we may assume that $\mu\in C_c^\infty(\mathbb{R}^d)$, and hence $\pi_e\mu\in C_c^\infty(e_1^\perp\times\cdots\times e_\ell^\perp)$ for $e\in S^{d_1-1}\times\cdots\times S^{d_{\ell}-1}$. From Lemma \ref{projlem1} it follows that
\begin{align}
&\int \|P^{(\ell)}_y\mu\|^p_{L^p(S^{d_1-1}\times\cdots\times S^{d_{\ell}-1})}\,d\nu(y) \label{eq:projlem}\\
\lesssim & \int \|P^{(\ell)}_y\mu_y\|^p_{L^p(S^{d_1-1}\times\cdots\times S^{d_{\ell}-1})}\,d\nu(y) = \int_{S^{d_1-1}\times\cdots\times S^{d_{\ell}-1}} \|\pi_e\mu\|_{L^p(\pi_e\nu)}^p d\mathcal{H}^{d-\ell}(e)\,. \notag 
\end{align}

Next, for fixed $e\in S^{d_1-1}\times\cdots\times S^{d_{\ell}-1}$, we estimate $\|\pi_e\mu\|_{L^p(\pi_e\nu)}$. By duality, we can choose non-negative $f$ with $\|f\|_{L^q(\pi_e\nu)}=1$ and $q=p'$ such that 
\begin{align}
&\|\pi_e\mu\|_{L^p(\pi_e\nu)} = \int_{e_1^\perp\times\cdots\times e_\ell^\perp} \pi_e\mu\cdot f d\pi_e\nu \label{eq:dual}\\
\leq & \left(\int_{e_1^\perp\times\cdots\times e_\ell^\perp} |\widehat{\pi_e \mu}(\xi)|^2\prod_{i=1}^\ell |\xi_i|^\delta\,d\xi \right)^{1/2}\left(\int_{e_1^\perp\times\cdots\times e_\ell^\perp} |\widehat{f d\pi_e\nu}(\xi)|^2\prod_{i=1}^\ell |\xi_i|^{-\delta}\,d\xi\right)^{1/2}. \notag
\end{align}
Moreover,
\begin{align}
&\int_{e_1^\perp\times\cdots\times e_\ell^\perp} |\widehat{f d\pi_e\nu}(\xi)|^2\prod_{i=1}^\ell |\xi_i|^{-\delta}\,d\xi \label{eq:holder}\\
\approx & \int \int f(z)f(w) \prod_{i=1}^\ell |z_i-w_i|^{\delta-d_i+1}\,d\pi_e\nu(z) d\pi_e\nu(w) \notag \\
\leq & \left( \int \int \prod_{i=1}^\ell |z_i-w_i|^{p(\delta-d_i+1)}\,d\pi_e\nu(z) d\pi_e\nu(w)\right)^{1/p}\,, \notag
\end{align}
where the second inequality follows from H\"older's inequality.

Now, in view of \eqref{eq:projlem}, by duality again, we can choose non-negative $g$ with $\|g\|_{L^q(S^{d_1-1}\times\cdots\times S^{d_{\ell}-1})}=1$ and $2<q=p'$ such that
\begin{align*}
&\left(\int \|P^{(\ell)}_y\mu\|^p_{L^p(S^{d_1-1}\times\cdots\times S^{d_{\ell}-1})}\,d\nu(y)\right)^{1/p} \\
\lesssim & \int_{S^{d_1-1}\times\cdots\times S^{d_{\ell}-1}} \|\pi_e\mu\|_{L^p(\pi_e\nu)} g(e) d\mathcal{H}^{d-\ell}(e) \leq A^{1/2}B^{1/2}\,,
\end{align*}
where the second inequality follows from the estimate of $\|\pi_e\mu\|_{L^p(\pi_e\nu)}$ as in \eqref{eq:dual} and \eqref{eq:holder}, and 
\begin{equation*}
    A=\int_{S^{d_1-1}\times\cdots\times S^{d_{\ell}-1}} \int_{e_1^\perp\times\cdots\times e_\ell^\perp} |\widehat{\pi_e \mu}(\xi)|^2\prod_{i=1}^\ell |\xi_i|^\delta\,d\xi
    d\mathcal{H}^{d-\ell}(e),
\end{equation*}
\begin{equation*}
    B=\int_{S^{d_1-1}\times\cdots\times S^{d_{\ell}-1}}  \left( \int \int \prod_{i=1}^\ell |z_i-w_i|^{p(\delta-d_i+1)}\,d\pi_e\nu(z) d\pi_e\nu(w)\right)^{1/p} g(e)^2 d\mathcal{H}^{d-\ell}(e)\,.
\end{equation*}
To complete the proof of Theorem \ref{thm: multipara proj0}, we will show that $A\lesssim C_\alpha(\mu)^2$ and $B\lesssim C_\beta(\nu)^{2/p}$.

Note that for $(\xi_1,\cdots,\xi_\ell)\in e_1^\perp\times\cdots\times e_\ell^\perp$,
\begin{align*}
&(\pi_{e_1}\times\cdots\times\pi_{e_\ell}\mu)^\wedge (\xi_1,\cdots,\xi_\ell)\\
= & (\textrm{Id}_{\mathbb{R}^{d_1}}\times\pi_{e_2}\times\cdots\times\pi_{e_\ell}\mu)^\wedge (\eta_1,\xi_2,\cdots,\xi_\ell) = \cdots = \widehat{\mu} (\eta_1,\cdots,\eta_\ell)\,,
\end{align*}
where $\eta_i=\xi_i$ is viewed as a point in $\mathbb{R}^{d_i}$. Hence, by applying Lemma \ref{projlem2} repeatedly in each variable and then applying Lemma \ref{projlem3}, we get that
\begin{equation} \label{A}
A\lesssim \int_{\mathbb{R}^d} |\widehat{\mu}(\eta)|^2\prod_{i=1}^{\ell} |\eta_i|^{\delta-1}\,d\eta \lesssim C_\alpha(\mu)^2\,,
\end{equation}
provided that $d_i+\delta-1<\alpha-d+d_i, \forall i=1,\cdots,\ell$, i.e. $\alpha>d-1+\delta$, which indeed holds by our choice of $\delta$. 

To estimate $B$, we first apply H\"older's inequality
$$
B^p \leq \int_{S^{d_1-1}\times\cdots\times S^{d_{\ell}-1}}  \int \int \prod_{i=1}^\ell |z_i-w_i|^{p(\delta-d_i+1)}\,d\pi_e\nu(z) d\pi_e\nu(w) g(e)^p d\mathcal{H}^{d-\ell}(e)
$$
and observe that
\begin{align*}
 &\int_{(B^{d_1}(x_1,r_1)\times\cdots\times B^{d_\ell}(x_\ell,r_\ell))\cap(S^{d_1-1}\times\cdots\times S^{d_\ell-1})} g(e)^p d\mathcal{H}^{d-\ell}(e)\\
 \leq & \left(\int_{S^{d_1-1}\times\cdots\times S^{d_\ell-1}} g(e)^q d\mathcal{H}^{d-\ell}(e)\right)^{p/q} \left(\mathcal{H}^{d-\ell} (B^{d_1}(x_1,r_1)\times\cdots\times B^{d_\ell}(x_\ell,r_\ell))\right)^{2-p}\\
 \lesssim & \prod_{i=1}^\ell r_i^{(d_i-1)(2-p)}\,.
\end{align*}
Therefore, by Lemma \ref{projlem4} and Lemma \ref{projlem3},
\begin{equation}\label{B}
 B^p\lesssim \int_{\mathbb{R}^d}\int_{\mathbb{R}^d} \frac{d\nu(x)d\nu(y)}{\prod_{i=1}^\ell |x_i-y_i|^{p(d_i-1-\delta)}} \lesssim C_\beta(\nu)^2\,,
\end{equation}
as desired, provided that 
$$p(d_i-1-\delta)<(d_i-1)(2-p), \text{ i.e. } p<\frac{2(d_i-1)}{2(d_i-1)-\delta}, \,\,\forall i=1,\cdots,\ell\,,
$$
and
$$p(d_i-1-\delta)<\beta -d + d_i, \text{ i.e. } p<\frac{\beta-d+d_i}{d_i-1-\delta}, \,\,\forall i=1,\cdots,\ell\,,
$$
which is indeed the case by our choice of $p$ as in \eqref{range of p}.
\end{proof}

When $\alpha\leq d-1$, Theorem \ref{thm: multipara proj0} is not applicable. However, we will show below that by combining with orthogonal projections, the multiparameter radial projection result does hold true on a set of product subspaces of certain dimensions depending on $\alpha$. Here is some notation.

Fix $\vec{d}=(d_1,\cdots,d_\ell)$ and $\vec{k}=(k_1,\cdots,k_\ell)$ with $\ell\geq 1$ and $2\leq k_i \leq d_i$. Let $d=d_1+\cdots d_\ell$ and $k=k_1+\cdots+k_\ell$. Denote 
$$
V=(V_1,\cdots,V_\ell)\in G(d_1,k_1)\times\cdots\times G(d_\ell,k_\ell)=:G,
$$
and
$$
d\gamma(V)=d\gamma_{d_1,k_1}(V_1)\cdots d\gamma_{d_\ell,k_\ell}(V_\ell).
$$ 
Let $\pi_i:\mathbb{R}^d\to\mathbb{R}^{d_i}$, $\pi_{V_i}:\mathbb{R}^{d_i}\to V_i$ and
$$
\pi_V=\pi_{V_1}\times\cdots\times\pi_{V_\ell}: \mathbb{R}^d=\mathbb{R}^{d_1}\times\cdots\times\mathbb{R}^{d_\ell}\to V_1 \times\cdots\times V_\ell =\mathbb{R}^k\,,
$$
be orthogonal projections. Given $V=(V_1,\cdots,V_\ell)\in G$ and $w=(w_1,\cdots,w_\ell)\in V_1\times\cdots\times V_\ell$, define the $\ell$-parameter radial projection
$$
P^{(\ell)}_w: V_1\times\cdots\times V_\ell \setminus \{z: z_i=w_i \text{ for some } i=1,\cdots,\ell\} \to S^{k_1-1}\times \cdots \times S^{k_\ell-1}
$$ 
by
$$
P^{(\ell)}_w (z) = \left(\frac{z_1-w_1}{|z_1-w_1|},\cdots, \frac{z_\ell-w_\ell}{|z_\ell-w_\ell|}\right)\,.
$$ 

\begin{theorem}[Multiparameter radial projection theorem]\label{thm: multipara proj} Let $M:=\max\{d-d_i+k_i-1:i=1,\cdots \ell\}$. Suppose $\alpha> M$, $\beta>2M-\alpha$ and $\beta>M-(k_i-1), \forall i=1,\cdots,\ell$. Then there exists $p=p(\alpha,\beta)>1$ such that the following holds. Suppose that $\mu, \nu \in \mathcal{M}(\mathbb{R}^d)$ with 
\begin{enumerate}
\item $\pi_i({\rm supp}\,\mu) \cap \pi_i({\rm supp}\,\nu) = \varnothing$, $\forall i=1,\ldots, \ell$;
\item $\mu(B(x,r))\leq  C_\alpha(\mu) r^\alpha$, $\nu(B(x,r))\leq C_\beta(\nu)r^\beta$, $\forall x\in \mathbb{R}^d$, $\forall r>0$. 
\end{enumerate}
Let $G'=\{V=(V_1,\cdots,V_\ell)\in G : \pi_{V_i}\circ\pi_i({\rm supp}\,\mu)\cap\pi_{V_i}\circ\pi_i({\rm supp}\,\nu)=\varnothing, \forall i=1,\cdots,\ell\}$. Then,  
\[
\int_{G'} \left(\int_{V_1\times\cdots\times V_\ell} \|P^{(\ell)}_w \pi_V\mu\|^p_{L^p(S^{k_1-1}\times\cdots\times S^{k_{\ell}-1})}\,d\pi_V \nu(w) \right)^{1/p}d\gamma(V) \lesssim C_\alpha(\mu) C_\beta(\nu)^{1/p},
\]
where the implicit constant depends only on $d$ and the diameter of ${\rm supp}\,\mu \cup {\rm supp}\,\nu$.
\end{theorem}

\begin{proof}
By the conditions that $\alpha> M$, $\beta>2M-\alpha$ and $\beta>M-(k_i-1), \forall i=1,\cdots,\ell$, we can choose $\delta>0$ such that  $\alpha-M>\delta>M-\beta$ and $k_i-1>\delta, \forall i=1,\cdots,\ell$. We'll prove that
\[
\int_{G'} \left(\int_{V_1\times\cdots\times V_\ell} \|P^{(\ell)}_w \pi_V\mu\|^p_{L^p(S^{k_1-1}\times\cdots\times S^{k_{\ell}-1})}\,d\pi_V \nu(w) \right)^{1/p}d\gamma(V) \lesssim C_\alpha(\mu) C_\beta(\nu)^{1/p},
\]
whenever 
\begin{equation} \label{rangeofp}
1<p<\min\left\{2, \frac{2(k_i-1)}{2(k_i-1)-\delta}\,, \frac{\beta-d+d_i}{k_i-1-\delta}: i=1,\cdots, \ell \right\}.
\end{equation}
Note that the choice of $\delta$ guarantees that the right hand side of \eqref{rangeofp} lies in $(1,2]$, so the range of $p$ is nonempty. By a similar limiting argument as in the proof of Theorem \ref{thm: multipara proj0}, we may assume that $\mu\in C_c^\infty(\mathbb{R}^d)$. 

Following the proof of Theorem \ref{thm: multipara proj0} up to \eqref{A} and \eqref{B}, we can tell that
$$
\left(\int_{V_1\times\cdots\times V_\ell} \|P^{(\ell)}_w \pi_V\mu\|^p_{L^p(S^{k_1-1}\times\cdots\times S^{k_{\ell}-1})}\,d\pi_V \nu(w) \right)^{1/p} \lesssim A^{1/2}B^{1/2}\,,
$$
where 
\begin{equation} \label{A'}
A\lesssim \int_{V_1\times\cdots\times V_\ell} |\widehat{\pi_V\mu}(\eta)|^2\prod_{i=1}^{\ell} |\eta_i|^{\delta-1}\,d\eta\,,
\end{equation}
and 
\begin{equation} \label{B'}
B^p\lesssim \int_{V_1\times\cdots\times V_\ell} \frac{|\widehat{\pi_V\nu}(\eta)|^2}{\prod_{i=1}^{\ell} |\eta_i|^{k_i-p(k_i-1-\delta)}}\,d\eta\,,
\end{equation}
provided that $p(k_i-1-\delta)<(k_i-1)(2-p)$, i.e. $p<\frac{2(k_i-1)}{2(k_i-1)-\delta}, \forall i=1,\cdots,\ell$, which indeed holds by our choice of $p$.

Therefore,
\begin{align*}
&\int_{G'} \left(\int_{V_1\times\cdots\times V_\ell} \|P^{(\ell)}_w \pi_V\mu\|^p_{L^p(S^{k_1-1}\times\cdots\times S^{k_{\ell}-1})}\,d\pi_V \nu(w) \right)^{1/p}d\gamma(V) \\
\lesssim & \left(\int_G A d\gamma(V)\right)^{1/2}\left(\int_G B d\gamma(V)\right)^{1/2}\,,
\end{align*}
and hence it suffices to prove that 
$$
\int_G A d\gamma(V) \lesssim C_\alpha(\mu)^2 \quad \text{ and } \quad \int_G B d\gamma(V) \lesssim C_\beta(\nu)^{2/p}.
$$

Indeed, because of \eqref{A'} and \eqref{B'}, by applying Lemma \ref{projlem2} repeatedly in each variable and then applying Lemma \ref{projlem3}, we get that
$$
\int_G A d\gamma(V) \lesssim \int_{\mathbb{R}^d} \frac{|\widehat{\mu}(\xi)|^2}{\prod_{i=1}^\ell |\xi_i|^{d_i-k_i+1-\delta}} \,d\xi \lesssim C_\alpha(\mu)^2\,,
$$
provided that $k_i-1+\delta<\alpha-d+d_i$, i.e. $\alpha>d-d_i+k_i-1+\delta, \forall i=1,\cdots,\ell$, which is indeed the case by the choice of $\delta$, and
$$
\left(\int_G B d\gamma(V)\right)^p \leq \int_G B^p d\gamma(V) \lesssim \int_{\mathbb{R}^d} \frac{|\widehat{\nu}(\xi)|^2}{\prod_{i=1}^\ell |\xi_i|^{d_i-p(k_i-1-\delta)}} \,d\xi \lesssim C_\beta(\nu)^2\,,
$$
provided that $p(k_i-1-\delta)<\beta-d+d_i$, i.e. $p<\frac{\beta-d+d_i}{k_i-1-\delta}, \forall i=1,\cdots,\ell$, which is guaranteed by the choice of $p$. This completes the proof.
\end{proof}

With Theorem \ref{thm: multipara proj} in tow, we are now ready to prove Lemma \ref{lem: step1main}.

\begin{proof}[Proof of Lemma \ref{lem: step1main}]
%First, it is easy to see from the definition of the bad regions that (\ref{eqn: Bad12}) follows immediately from (\ref{eqn: Bad1Bad2}). Indeed, observe that $\text{Bad}'_{j_1,j_2}\subset \text{Bad}_{j_1}\cap \text{Bad}_{j_2}$, hence $\mu_1\times \mu_2(\text{Bad}'_{j_1,j_2})\leq \mu_1\times \mu_2(\text{Bad}_{j_i})\lesssim R_{j_i}^{-4\delta}$, $i=1,2$, which implies (\ref{eqn: Bad12}).

%We now turn to the proof of (\ref{eqn: Bad1Bad2}). Since the two bounds are symmetric, here we only present the proof for $\text{Bad}_{j_2}$.

Recall from the construction of sets $E_1, E_2$, that for each $i=1,\cdots, \ell$, there is a pre-selected $k_i$-dimensional subspace $W_i\subset \mathbb{R}^{d_i}$ (with $k_i$ equal to $d_i-\frac{d_{\min}}{2}+1$ when $d_{\min}$ is even, $d_i-\frac{d_{\min}}{2}+\frac{1}{2}$ when $d_{\min}$ is odd) such that the projections of $E_1, E_2$ onto each $W_i$ are separated by distance $\gtrsim 1$.

For each $j=1,2$, there exists a Frostman measure $\mu_j$ supported on $E_j$ satisfying
\[
\mu_j(B(x,r))\lesssim r^\alpha,\quad \forall x\in \mathbb{R}^d,\, \forall r>0.
\]

For each $1\leq i\leq \ell$, since $\pi_{W_i}\circ\pi_i(E_1), \pi_{W_i}\circ\pi_i(E_2)$ are well separated, one has that ${\rm dist}(\pi_{V_i}\circ\pi_i(E_1), \pi_{V_i}\circ\pi_i(E_2))\gtrsim 1$ for all $V_i$ in a small enough neighborhood of $W_i$ in $G(d_i, k_i)$, the Grassmanian of $k_i$-dimensional subspaces in $\mathbb{R}^{d_i}$. 

Since $\alpha>d-\frac{d_{\min}}{2}$, the pair of measures $\mu_1, \mu_2$ satisfy the conditions in Theorem \ref{thm: multipara proj}. Therefore, applying Theorem \ref{thm: multipara proj}, there must exist some $V_i\in G(d_i, k_i)$ in a small neighborhood of $W_i$, $i=1,\cdots,\ell$, such that the inner integral in the concluded estimate in Theorem \ref{thm: multipara proj} is finite, in other words, for some $p>1$,
\begin{equation}\label{eqn: proj finite}
\int_{V_1\times\cdots\times V_\ell} \|P^{(\ell)}_w \pi_V\mu_2\|^p_{L^p(S^{k_1-1}\times\cdots\times S^{k_{\ell}-1})}\,d\pi_V \mu_1(w)<+\infty.
\end{equation}

Our goal is to estimate the set $\text{Bad}_{j_1,\cdots,j_k}^{T_{m+1},\cdots, T_k}$ by a properly defined bad region in $V_1\times\cdots\times V_\ell$ and to apply (\ref{eqn: proj finite}).

More precisely, note that for each $i=1,\cdots,m$ and each $j_i\geq 1$, each tube $T_i\in\mathbb{T}^i_{j_i}$ under consideration is projected by $\pi_{V_i}$ to a tube of comparable dimensions (with side length $\sim 1$ in the long direction, and $\sim R_{j_i}^{-\frac{1}{2}+\delta}$ in the rest of the directions). This is because that if $T_i$ appears in a product bad tube, then by definition of $\text{Bad}_{j_1,\cdots,j_k}^{T_{m+1},\cdots, T_k}$ it must intersect both $\pi_i(E_1)$ and $\pi_i(E_2)$, where $\pi_i:\,\mathbb{R}^d\rightarrow \mathbb{R}^{d_i}$ denotes the orthogonal projection. Therefore, the collection $\mathbb{T}_{j_i}^i$ gives rise to a collection $\tilde{\mathbb{T}}^i_{j_i}$ of tubes in $V_i$. We also denote $\tilde{\mathbb{T}}^i=\cup_{j_i\geq 1} \tilde{\mathbb{T}}^i_{j_i}$.

For every $(\tilde{T}_1,\cdots, \tilde{T}_m)\in \tilde{\mathbb{T}}^1_{j_1}\times \cdots\times \tilde{\mathbb{T}}^m_{j_m}$, define it to be \emph{bad} if
\[
\begin{split}
&\pi_V\mu_2(4\tilde{T}_1\times \cdots\times 4\tilde{T}_m\times B_1^{V_{m+1}}\times\cdots\times B_1^{V_\ell})\geq\\ & \quad C_0(R_{j_{m+1}},\cdots, R_{j_k})\cdot\begin{cases} \prod_{i=1}^m R_{j_{i}}^{-(\frac{d_{i}}{2}-\frac{d_{\min}}{4}) + c(\alpha) \delta}, & d_{\min} \text{ is even}; \\ \prod_{i=1}^m R_{j_{i}}^{-(\frac{d_{i}}{2}-\frac{d_{\min}}{4}-\frac{1}{4}) + c(\alpha) \delta}, & d_{\min} \text{ is odd}. \end{cases},
\end{split}
\]where $B_1^{V_i}$ denotes the unit ball in $V_i$ centered at the origin.

Then, it is easy to see that for any given $(T_{m+1},\cdots, T_k)\in \mathbb{T}^{m+1}_{j_{m+1}}\times\cdots\times \mathbb{T}^k_{j_k}$,
\[
\mu_1\times \mu_2 (B_{j_1,\cdots, j_k}^{T_{m+1},\cdots, T_{k}})\leq \pi_V\mu_1\times \pi_V\mu_2(\widetilde{\text{Bad}}_{j_1,\cdots, j_m}),
\]with the bad region defined as
\[
\begin{split}
\widetilde{\text{Bad}}_{j_1,\cdots, j_m}:=&\big\{(\tilde{y},\tilde{x})\in \pi_V(E_1)\times \pi_V(E_2):\, \exists (\tilde{T}_1,\cdots, \tilde{T}_m) \in \tilde{\mathbb{T}}^1_{j_1}\times \cdots\times \tilde{\mathbb{T}}^m_{j_m} \text{ s.t. }\\
&\quad \tilde{T}_1\times \cdots \times \tilde{T}_m \text{ is bad},\, \tilde{y}_i, \tilde{x}_i\in 2\tilde{T}_i,\,\forall i=1,\cdots, m  \big\}.
\end{split}
\]

Indeed, if $(T_1,\cdots,T_m)\in \mathbb{T}^1_{j_1}\times \cdots \mathbb{T}^m_{j_m}$ is such that $T_1\times\cdots\times T_k$ is ($1,\cdots, k$)-bad, then there holds
\[
\begin{split}
&\pi_V \mu_2(4\pi_V(T_1)\times\cdots\times 4\pi_V(T_m)\times B_1^{V_{m+1}}\times\cdots\times B_1^{V_\ell})\\
\geq & \mu_2(4T_1\times \cdots \times4 T_m\times B_1^{d_{m+1}}\times\cdots\times B_1^{d_\ell})\\
\geq & \mu_2(4T_1\times\cdots \times 4T_k\times B_1^{d_{k+1}}\times\cdots\times B_1^{d_\ell}).
\end{split}
\]Hence, the image of $T_1\times\cdots \times T_m$ under the projection $\pi_V$ is contained in some bad $\tilde{T}_1\times\cdots\times \tilde{T}_m$, if one chooses 
\[
C_0(R_{j_{m+1}},\cdots, R_{j_k})=\begin{cases} \prod_{i=m+1}^k R_{j_{i}}^{-(\frac{d_{i}}{2}-\frac{d_{\min}}{4}) + c(\alpha) \delta}, & d_{\min} \text{ is even}; \\ \prod_{i=m+1}^k R_{j_{i}}^{-(\frac{d_{i}}{2}-\frac{d_{\min}}{4}-\frac{1}{4}) + c(\alpha) \delta}, & d_{\min} \text{ is odd}. \end{cases}.
\]It thus suffices to estimate $\pi_V\mu_1\times \pi_V\mu_2(\widetilde{\text{Bad}}_{j_1,\cdots, j_m})$.

Write
\[
\begin{split}
\pi_V\mu_1\times \pi_V\mu_2(\widetilde{\text{Bad}}_{j_1,\cdots, j_m})=&\int \pi_V\mu_2(\widetilde{\text{Bad}}_{j_1,\cdots, j_m}(\tilde{y})) \, d\pi_V\mu_1(\tilde{y})\\
\leq & \int P_{\tilde{y}}^{(\ell)}\pi_V \mu_2 \left(P_{\tilde{y}}^{(\ell)}\big(\widetilde{\text{Bad}}_{j_1,\cdots, j_m}(\tilde{y}) \big) \right) \, d\pi_V \mu_1(\tilde{y}),
\end{split}
\]where
\[
\widetilde{\text{Bad}}_{j_1,\cdots, j_m}(\tilde{y})=\left(\bigcup_{\substack{(\tilde{T}_1\times\cdots\times \tilde{T}_m)\in \tilde{\mathbb{T}}^1_{j_1}\times \cdots\times \tilde{\mathbb{T}}^m_{j_m} \text{ bad}\\ \tilde{y}_i\in 2\tilde{T}_i,\,i=1,\cdots, m}} 2\tilde{T}_1\times\cdots\times 2\tilde{T}_m\right)\times \left( \prod_{i=m+1}^\ell B_1^{V_i}\right).
\]
Then, by H\"older's inequality and the multiparameter radial projection estimate (\ref{eqn: proj finite}), we get that
\begin{align*}
&\pi_V\mu_1\times \pi_V\mu_2(\widetilde{\text{Bad}}_{j_1,\cdots, j_m})\\
\leq & \sup_{\tilde y} \left| P_{\tilde{y}}^{(\ell)}\big(\widetilde{\text{Bad}}_{j_1,\cdots, j_m}(\tilde{y}) \big) \right|^{1-\frac 1p} \int \|P^{(\ell)}_{\tilde y} \pi_V\mu_2\|_{L^p} \,d\pi_V \mu_1(\tilde y)\\
\lesssim & \sup_{\tilde y} \left| P_{\tilde{y}}^{(\ell)}\big(\widetilde{\text{Bad}}_{j_1,\cdots, j_m}(\tilde{y}) \big) \right|^{1-\frac 1p}\,.
\end{align*}

It remains to estimate $\left| P_{\tilde{y}}^{(\ell)}\big(\widetilde{\text{Bad}}_{j_1,\cdots, j_m}(\tilde{y}) \big) \right|$. For every $\tilde{y}\in V_1\times\cdots\times V_\ell$ and $\tilde{T}_i\in \tilde{\mathbb{T}}^i_{j_i}$, let $A(\tilde{T}_i)$ be the cap on the sphere $S^{k_i-1}\subset V_i$ whose center corresponds to the direction of the long axis of $\tilde{T}_i$ and with radius $\sim R_{j_i}^{-\frac{1}{2}+\delta}$. Since ${\rm dist}(\pi_{V_i}\circ \pi_i(E_1), \pi_{V_i}\circ \pi_i(E_2))\gtrsim 1$, one has $P_{\tilde{y}_i}(\pi_{V_i}\circ \pi_i (E_2)\cap 4\tilde{T}_i)\subset A(\tilde{T}_i)$, where $P_{\tilde{y}_i}$ denotes the one-parameter radial projection map centered at $\tilde{y}_i\in V_i$. Similarly, it is easy to see that the product cap $A(\tilde{T}_1)\times \cdots \times A(\tilde{T}_m)\times S^{k_{m+1}-1}\times\cdots \times S^{k_\ell-1}$ contains the image of $4\tilde{T}_1\times\cdots\times 4\tilde{T}_m\times B_1^{V_{m+1}}\times \cdots \times B_1^{V_\ell}$ under the radial projection $P_{\tilde{y}}^{(\ell)}$.

Therefore, one can cover $P_{\tilde{y}}^{(\ell)}\big(\widetilde{\text{Bad}}_{j_1,\cdots, j_m}(\tilde{y}) \big)$ by product caps of the form $A(\tilde{T}_1)\times \cdots \times A(\tilde{T}_m)\times S^{k_{m+1}-1}\times\cdots \times S^{k_\ell-1}$, where each $\tilde{T}_1\times\cdots\times \tilde{T}_m$ is bad. For any such product cap, there holds
\[
\begin{split}
&P_{\tilde{y}}^{(\ell)}\pi_V \mu_2\left(A(\tilde{T}_1)\times \cdots \times A(\tilde{T}_m)\times S^{k_{m+1}-1}\times\cdots \times S^{k_\ell-1} \right) \\
\geq &\pi_V \mu_2(4\tilde{T}_1\times\cdots\times 4\tilde{T}_m\times B_1^{V_{m+1}}\times \cdots \times B_1^{V_\ell})\\
\geq &C_0(R_{j_{m+1}},\cdots, R_{j_k})\cdot\begin{cases} \prod_{i=1}^m R_{j_{i}}^{-(\frac{d_{i}}{2}-\frac{d_{\min}}{4}) + c(\alpha) \delta}, & d_{\min} \text{ is even}; \\ \prod_{i=1}^m R_{j_{i}}^{-(\frac{d_{i}}{2}-\frac{d_{\min}}{4}-\frac{1}{4}) + c(\alpha) \delta}, & d_{\min} \text{ is odd}. \end{cases}
\end{split}
\]

By the Vitali covering lemma, there exists a disjoint subcollection $\mathcal{C}$ of such product caps so that their dilations by a constant (say, $5$) cover $P_{\tilde{y}}^{(\ell)}\big(\widetilde{\text{Bad}}_{j_1,\cdots, j_m}(\tilde{y}) \big)$. Hence, the total number of disjoint product caps in this covering is bounded by
\[
\#\mathcal{C}\leq C_0(R_{j_{m+1}},\cdots, R_{j_k})^{-1}\cdot\begin{cases} \prod_{i=1}^m R_{j_{i}}^{\frac{d_{i}}{2}-\frac{d_{\min}}{4} - c(\alpha) \delta}, & d_{\min} \text{ is even}; \\ \prod_{i=1}^m R_{j_{i}}^{\frac{d_{i}}{2}-\frac{d_{\min}}{4}-\frac{1}{4} - c(\alpha) \delta}, & d_{\min} \text{ is odd}. \end{cases}
\]Therefore, 
\[
\begin{split}
\left| P_{\tilde{y}}^{(\ell)}\big(\widetilde{\text{Bad}}_{j_1,\cdots, j_m}(\tilde{y}) \big) \right| 
\lesssim &\left( \prod_{i=1}^m R_{j_i}^{(-\frac{1}{2}+\delta)(k_i-1)}\right)\cdot\#\mathcal{C}\\
\lesssim & C(R_{j_{m+1}},\cdots, R_{j_k})\prod_{i=1}^m R_{j_i}^{-(c(\alpha)-d_i+\frac{d_{\min}}{2})\delta},
\end{split}
\]where $|\cdot|$ denotes the product surface measure on $S^{k_1-1}\times\cdots\times S^{k_\ell-1}$. Then the desired estimate (\ref{eqn: bad mu}) follows by taking $c(\alpha)$ sufficiently large.

%Assume $d_{\min}$ is even, and recall that we have fixed a choice of $(d_2-\frac{d_{\min}}{2}+1)$ dimensional subspace $V_2\subset \mathbb{R}^{d_2}$ such that the projected measure $\tilde{\tilde{\mu}}_i:=(\pi_{V_2})_\ast(\tilde{\mu}_i)$ satisfies $I_\beta(\tilde{\tilde{\mu}}_i)<\infty$, $\forall \beta<\alpha-d_1$. Here $\pi_{V_2}$ is the orthogonal projection from $\mathbb{R}^{d_2}$ onto $V_2$. The key point here is that the following (rephrased version of the) radial projection theorem of Orponen \cite{O17b} applies to $\tilde{\tilde{\mu}}_1, \tilde{\tilde{\mu}}_2$ in $V_2$, since $\alpha-d_1>d_2-\frac{d_{\min}}{2}={\rm dim}(V_2)-1$.

%\begin{theorem}\label{thm: orponen}{\cite[Orponen]{O17b}}
%For every $\beta > n-1$ there exists $p(\beta)> 1$ so that the following holds.  Suppose that $\mu$ and $\nu$ are compactly supported Radon measures on $\mathbb{R}^n$ with disjoint support  and that $I_{\beta}(\mu )<\infty $, $I_\beta(\nu)<\infty$.  Then
	%\[
	%\int \| P_y \mu \|_{L^p}^p d \nu(y) < \infty,
	%\]where $P_y: \mathbb{R}^n \setminus \{ y \} \rightarrow S^{n-1}$ is the radial projection map defined as $P_y(x) = \frac{x-y}{|x-y|}$.
%\end{theorem}

\end{proof}

\section{Step 2: proof of (\ref{eqn: mainest2})}\label{sec: dec}
\setcounter{equation}0

In this section, we prove (\ref{eqn: mainest2}), which will conclude the proof of our main theorem. 

Recall the setup. Let $d=d_1+\cdots+d_\ell$ with $\ell\geq1$ and $d_{\min}=\min\{d_i:i=1,\cdots,\ell\}\geq 2$. Let $E_1, E_2 \subset B^d(0,1)$ with ${\rm dist}(\pi_i(E_1),\pi_i(E_2))\gtrsim 1, \forall i=1,\cdots,\ell$. And $\mu_j$ is a  probability measure supported on $E_j$ such that 
$$
\mu_j(B(x,r))\lesssim r^\alpha, \quad \forall x\in \mathbb{R}^d, \forall r>0, \quad {\rm where} \, j=1,2.
$$
The good part of $\mu_1$ with respect to $\mu_2$ is defined as
\[
\mu_{1,g}=\sum_{k=0}^{\ell} \sum_{1\leq i_1<\cdots<i_k\leq\ell} \bigg(\prod_{i\neq i_1,\cdots,i_k} M^i_0\bigg) \bigg(\sum_{\substack{(T_{i_1},\cdots,T_{i_k})\in \mathbb{T}^{i_1}\times\cdots\times\mathbb{T}^{i_k} \\ (i_1,\cdots,i_k){\textrm -good}}} M^{i_1}_{T_{i_1}}\cdots M^{i_k}_{T_{i_k}}\mu_1\bigg).
\]
We want to prove that 
\begin{equation}\label{L2G}
\int_{E_2}  \| d^x_*\mu_{1,g}\|_{L^2}^2 d \mu_2(x) < \infty\,,
\end{equation}
if 
\[
\alpha>\begin{cases} d-\frac{d_{\min}}{2}+\frac{1}{4},& d_{\min} \text{ is even},\\ d-\frac{d_{\min}}{2}+\frac{1}{4}+\frac{1}{4d_{\min}},& d_{\min} \text{ is odd}.\end{cases}
\]

By the definition of a pushforward measure, we have 
\[
d^x_*\mu_{1,g}(t_1,\cdots,t_\ell)=t_1^{d_1-1}\cdots t_\ell^{d_\ell-1}\mu_{1,g}*^{(1)}\sigma_{t_1}*^{(2)}\sigma_{t_2}\cdots *^{(\ell)} \sigma_{t_\ell} (x)\,,
\]
where $\sigma_{t_i}$ denotes the normalized surface measure on $t_i S^{d_i-1}$ and $*^{(i)}$
stands for the convolution in the $i$-th variable. Because of the assumption that ${\rm dist}(\pi_i(E_1),\pi_i(E_2))\gtrsim 1, \forall i=1,\cdots,\ell$, we only need to consider $t_1,\cdots,t_\ell \sim 1$. Hence,
\begin{align}
&\int_{E_2}  \| d^x_*\mu_{1,g}\|_{L^2}^2 d \mu_2(x) \label{eq:dblint}\\
\sim & \int_{E_2} \int_{\mathbb{R}^\ell_+} |\mu_{1,g}*^{(1)}\sigma_{t_1}*^{(2)} \cdots *^{(\ell)} \sigma_{t_\ell} (x)|^2 t_1^{d_1-1}\cdots t_\ell^{d_\ell-1} \,dt\,d \mu_2(x) \notag\\
= & \int_{E_2} \int_{\mathbb{R}^\ell_+} |\mu_{1,g}*^{(1)}\widehat{\sigma_{r_1}}*^{(2)} \cdots *^{(\ell)} \widehat{\sigma_{r_\ell}} (x)|^2 r_1^{d_1-1}\cdots r_\ell^{d_\ell-1} \,dr\,d \mu_2(x) \notag\\
= &  \int_{\mathbb{R}^\ell_+} \int_{E_2} |\mu_{1,g}*^{(1)}\widehat{\sigma_{r_1}}*^{(2)} \cdots *^{(\ell)} \widehat{\sigma_{r_\ell}} (x)|^2 \,d \mu_2(x) \, r_1^{d_1-1}\cdots r_\ell^{d_\ell-1} \,dr\,, \notag
\end{align}
where the second equation follows by applying an $L^2$ identity of Liu \cite[Theorem 1.9]{Liu18} iteratively in each variable. For each fixed $r=(r_1,\cdots,r_\ell)\in \mathbb{R}^\ell_+$, we have the following multiparameter weighted Fourier extension estimate for the above inner integral over $E_2$. Denote $\vec{\sigma}_r=\sigma_{r_1}\otimes\cdots\otimes\sigma_{r_\ell}$.

\begin{lemma}\label{lem:MWFE}
Let $p=\max\left\{\frac{2(d_i+1)}{d_i-1}: i=1,\cdots,\ell\right\}=\frac{2(d_{\min}+1)}{d_{\min} -1}$. Let
$$
\eta_i=
\frac{2(d-\alpha-d_i)}{p}  -(\frac{d_i}{2}-\frac{d_{\min}}{4})(1-\frac{2}{p})
$$
when $d_{\min}$ is even, and
$$
\eta_i=
\frac{2(d-\alpha-d_i)}{p} -(\frac{d_i}{2}-\frac{d_{\min}}{4}-\frac 14)(1-\frac{2}{p})
$$
when $d_{\min}$ is odd. Then, for any $0<\alpha\leq d$ and any $\epsilon>0$ with $\epsilon+\eta_i<0, \forall i=1,\cdots,\ell$, there exists a constant $C_\epsilon$ such that the following holds. For any $r=(r_1,\cdots,r_\ell)\in \mathbb{R}^\ell_+$,

\begin{align}\label{eq:MWFE}
  &\int_{E_2} |\mu_{1,g}*^{(1)}\widehat{\sigma_{r_1}}*^{(2)} \cdots *^{(\ell)} \widehat{\sigma_{r_\ell}} (x)|^2 \,d \mu_2(x) \\
  \leq & C_\epsilon \bigg(\prod_{i=1}^\ell  \min\left(r_i^{\eta_i+\epsilon{-(d_i-1)}}, 1\right) \bigg) \int |\widehat{\mu_1}|^2\psi_r\,d\xi\,,  \notag
\end{align}
where
\[
\psi_r(\xi)=\prod_{i=1}^\ell \psi_{r_i}(\xi_i).
\]Here, each $\psi_{r_i}$ is a weight function that is $\sim 1$ on the annulus $r_i-1\leq |\xi_i|\leq r_i+1$ and decays off it. To be precise, we could take
$$
\psi_{r_i}(\xi_i)=\left(1+\left|r_i-|\xi_i|\right|\right)^{-100d_i}\,.
$$
\end{lemma}

Then, from \eqref{eq:dblint}, Lemma \ref{lem:MWFE}, and the inequality\footnote[2]{If in the integral over the interval $[0,1]$ the factor $r_i^{d_i-1}$ is replaced by $r_i^{\eta_i+\epsilon}$, then this inequality no longer holds in the case that $\eta_i<-1$. This is the reason that we put $\min(r_i^{\eta_i+\epsilon-(d_i-1)},1)$ instead of $r_i^{\eta_i+\epsilon-(d_i-1)}$ on the right hand side of the estimate in Lemma \ref{lem:MWFE}. 

A correction for the one parameter case in higher dimensions: in \cite[Lemma 4.1]{DIOWZ} we also need to deal with a similar issue, and the factor $r^{-\frac{d}{2(d+1)}-\frac{(d-1)\alpha}{d+1}+\epsilon-(d-1)}$ should be changed to $\min(r^{-\frac{d}{2(d+1)}-\frac{(d-1)\alpha}{d+1}+\epsilon-(d-1)},1)$. The new estimate still holds following the proof there.}

$$
 \int_0^1 r_i^{d_i-1}\psi_{r_i}(\xi_i)dr_i + \int_1^\infty r_i^{\eta_i+\epsilon} \psi_{r_i}(\xi_i)dr_i \lesssim |\xi_i|^{\eta_i+\epsilon}\,, 
$$
it follows that
$$
\int_{E_2}  \| d^x_*\mu_{1,g}\|_{L^2}^2 d \mu_2(x) \lesssim_\epsilon \int_{\mathbb{R}^d} |\widehat{\mu_1}(\xi)|^2 \prod_{i=1}^\ell |\xi_i|^{\eta_i+\epsilon}\,d\xi\,,
$$
where the right hand side is finite, by Lemma \ref{projlem3}, provided that $$d_i+\eta_i<\alpha-d+d_i, \textrm{ i.e. } \alpha>d+\eta_i,\,\, \forall i=1,\cdots,\ell.$$
By a direct calculation, $\alpha>d+\eta_i$ if and only if
\[
\alpha>\begin{cases} d-\frac{d_i}{2}+\frac{1}{4},& d_{\min} \text{ is even},\\ d-\frac{d_i}{2}+\frac{1}{4}+\frac{1}{4d_{\min}},& d_{\min} \text{ is odd},\end{cases}
\]
which concludes the desired result \eqref{L2G}.

It remains to prove Lemma \ref{lem:MWFE}. A crucial tool is a multiparameter refined decoupling theorem. 

Here is the setup. Let $d=d_1+\cdots+d_\ell$ with $\ell\geq1$ and $d_i\geq 2, \forall i=1,\cdots,\ell$.
Let $S_i \subset \mathbb{R}^{d_i}$ be a compact and strictly convex $C^2$ hypersurface with Gaussian curvature $\sim 1$. For any $\epsilon>0$, choose $0<\delta\ll \epsilon$. For any $R_i\geq 1$, decompose the $1$-neighborhood of $R_i S_i$ in $\mathbb{R}^{d_i}$ into blocks $\theta_i$ of dimensions $R_i^{1/2} \times \cdots \times R_i^{1/2} \times 1$.  For each $\theta_i$, let $\mathbb{T}_{\theta_i}$ be a finitely overlapping covering of $B^{d_i}(0,1)$ by tubes $T_i$ of dimensions $R_i^{-1/2 + \delta} \times\cdots\times R_i^{-1/2+\delta} \times 1$ with long axis perpendicular to $\theta_i$, and let $\mathbb{T}_i = \bigcup_{\theta_i} \mathbb{T}_{\theta_i}$. Each $T_i\in \mathbb{T}_i$ belongs to $\mathbb{T}_{\theta_i}$ for a single $\theta_i$, and we let $\theta(T_i)$ denote this $\theta_i$. Let $\mathbb{T}=\{T=T_1\times\cdots\times T_\ell: T_i\in\mathbb{T}_i\}$. For $T=T_1\times\cdots\times T_\ell\in\mathbb{T}$, denote $\theta(T)=\theta(T_1)\times\cdots\times\theta(T_\ell)$.
We say that $f$ is \emph{microlocalized} to $(T,\theta(T))$ if $f$ is essentially supported in $2T$ and $\widehat{f}$ is essentially supported in $2\theta(T)$.

\begin{theorem}[$\ell$-parameter refined decoupling]\label{thm:lRefDec}
Let $p\geq 2$. Let $f = \sum_{T \in \mathbb{W}} f_T$, where $\mathbb{W} \subset \mathbb{T}$ and $f_T$ is microlocalized to $(T, \theta(T))$.
Let $Y$ be a union of boxes $q=q_1\times\cdots\times q_\ell$ in $B_1^d$ each of which intersects at most $M$ product tubes $T \in \mathbb{W}$, where $q_i$'s are $R_i^{-1/2}$-cubes in $B_1^{d_i}$. Then,
\begin{enumerate}
\item $\forall \epsilon>0$, \[\| f \|_{L^p(Y)} \leq C_\epsilon \bigg(\prod_{i=1}^\ell R_i^{\gamma_i+\epsilon}\bigg) M^{\frac{1}{2} - \frac{1}{p}} \bigg(\sum_{T \in \mathbb{W}} \| f_T \|_{L^p}^p \bigg)^{1/p}.
\]
\item 
Suppose further that $\|f_T\|_{L^p}$ is roughly constant among all $T\in \mathbb{W}$ and denote $W=|\mathbb{W}|$. Then, $\forall\epsilon>0$,
\[\| f \|_{L^p(Y)} \leq C_\epsilon \bigg(\prod_{i=1}^\ell R_i^{\gamma_i+\epsilon}\bigg) \bigg(\frac MW\bigg)^{\frac{1}{2} - \frac{1}{p}} \bigg(\sum_{T \in \mathbb{W}} \| f_T \|_{L^p}^2 \bigg)^{1/2}.
\]
\end{enumerate}
Here
\[
\gamma_i=\begin{cases} 0,& 2\leq p\leq \frac{2(d_i+1)}{d_i-1},\\ \frac{d_i-1}{4}-\frac{d_i+1}{2p}, & p\geq  \frac{2(d_i+1)}{d_i-1}.\end{cases}
\]
\end{theorem}

Note that the power $\gamma_i$ of $R_i$ is the same as that in the classical decoupling theorem for the corresponding $p$. The two parts of Theorem \ref{thm:lRefDec} are equivalent: part (1) implies part (2) trivially; conversely, part (1) follows from part (2) combined with a dyadic pigeonholing about the quantity $\|f_T\|_{L^p}$.

The proof of Theorem \ref{thm:lRefDec} follows in the exact same way as that of its one-parameter analogue \cite[Corollary 4.3]{GIOW}, except that the use of the classical decoupling theorem is replaced by its iterated version. The details are left to interested readers.

\begin{proof}[Proof of Lemma \ref{lem:MWFE}]
Given $r=(r_1,\cdots,r_\ell)\in \mathbb{R}^\ell_+$, suppose $r_i>10 R_0$ for $i\in\{i_1,\cdots,i_k\}$ and $r_i \leq 10 R_0$ for $i\notin \{i_1,\cdots,i_k\}$. Here $0\leq k\leq \ell$ and $1\leq i_1<\cdots<i_k\leq\ell$. To ease the notation, in the following we'll consider the case $\{i_1,\cdots,i_k\}=\{1,\cdots,k\}$ (the computation for the most general case is similar). Then, by definition \eqref{DefG} of $\mu_{1,g}$, one has that $\mu_{1,g}*^{(1)}\widehat{\sigma_{r_1}}*^{(2)} \cdots *^{(\ell)} \widehat{\sigma_{r_\ell}}$ is equal to
$$
\sum_{\substack{R_{j_1}\sim r_1,\\
\cdots,R_{j_k}\sim r_k} } \bigg(\prod_{i=k+1}^\ell M^i_0\bigg) \bigg(\sum_{\substack{(T_1,\cdots,T_k)\in \mathbb{T}^1_{j_1}\times\cdots\times\mathbb{T}^k_{j_2} \\ (1,\cdots,k){\textrm -good}}} M^1_{T_1}\cdots M^k_{T_k}\mu_1\bigg)*^{(1)}\widehat{\sigma_{r_1}}*^{(2)} \cdots *^{(\ell)} \widehat{\sigma_{r_\ell}}.
$$

Fix $p=\max\left\{\frac{2(d_i+1)}{d_i-1}: i=1,\cdots,\ell\right\}=\frac{2(d_{\min}+1)}{d_{\min} -1}$. Let $\eta_1$ be a bump function adapted to the unit ball $B^d_1$ and define $f_{T_1\cdots T_k}$ to be
\[
\eta_1\bigg(\bigg(\prod_{i=k+1}^\ell M^i_0\bigg) \bigg(M^1_{T_1}\cdots M^k_{T_k}\mu_1\bigg)*^{(1)}\widehat{\sigma_{r_1}}*^{(2)} \cdots *^{(\ell)} \widehat{\sigma_{r_\ell}}\bigg).
\]
Denote $x=(x',x'')$, where $x'=(x_1,\cdots,x_k)$ and $x''=(x_{k+1},\cdots,x_\ell)$.
Then for fixed $x''$, $f_{T_1\cdots T_k}(\cdot,x'')$ is microlocalized in $(T,\theta(T))$, where $T=T_1\times\cdots\times T_k$ and $\theta(T)=\theta(T_1)\times\cdots\times\theta(T_k)$, following the same deduction in \cite[Section 5]{GIOW}. 

For the sake of convenience, we write $A\lessapprox B$ if $A\leq C_\epsilon\left(\prod_{i=1}^k R_i^{\epsilon}\right)B, \forall \epsilon>0$. When $r_{k+1},\cdots,r_\ell\leq 10 R_0$,  we want to prove that
\begin{equation*}
\int_{E_2} |\mu_{1,g}*^{(1)}\widehat{\sigma_{r_1}}*^{(2)} \cdots *^{(\ell)} \widehat{\sigma_{r_\ell}} (x)|^2 \,d \mu_2(x) \lessapprox \bigg(\prod_{i=1}^k r_i^{\eta_i-(d_i-1)}\bigg) \int |\widehat{\mu_1}|^2\psi_r\,d\xi\,,
\end{equation*}
where $\eta_i$ is as given in Lemma \ref{lem:MWFE}. 
Applying dyadic pigeonholing, one can find $\lambda$ such that
\begin{equation} \label{eq:dyala}
\int_{E_2} |\mu_{1,g}*^{(1)}\widehat{\sigma_{r_1}}*^{(2)} \cdots *^{(\ell)} \widehat{\sigma_{r_\ell}} (x)|^2 \,d \mu_2(x) \lessapprox \int |f_{ \lambda}(x)|^2\,d\mu_2(x),
\end{equation}
where
\[
f_{\lambda}=\sum_{T_1\times\cdots\times T_k \in \mathbb{W}_{\lambda}}f_{T_1\cdots T_k}.
\]
Here, we define $\mathbb{W}_{\lambda}$ to be
\begin{equation*}
\begin{aligned}
 \bigcup_{\substack{R_{j_1}\sim r_1,\\
\cdots,R_{j_k}\sim r_k}}{}
& \Big\{ T_1\times\cdots\times T_k: (T_1,\cdots, T_k) \in \mathbb{T}^1_{j_1}\times\cdots\times \mathbb{T}^k_{j_k} \text{ is ($1,\cdots,k$)-good},\\
& \qquad\qquad\qquad\qquad\qquad\| f_{T_1\cdots T_k} \|_{L^p(dx)} \sim \lambda \Big\}\,,
\end{aligned}
\end{equation*}
and denote $W=|\mathbb{W}_\lambda|$.

Next, we pigeonhole to obtain the region $Y$, as a union of boxes $q=q_1\times\cdots\times q_k$ each of which intersects $\sim M$ product tubes $T \in \mathbb{W}_\lambda$, where $q_i$'s are $R_i^{-1/2}$-cubes in $B_1^{d_i}$, such that
\begin{equation} \label{eq:dyaY}
\int |f_{ \lambda}(x)|^2\,d\mu_2(x) \lessapprox \int_{Y\times B_1^{d''}} |f_{ \lambda}(x)|^2\,d\mu_2(x)\,.
\end{equation}
Here $d''=d_{k+1}+\cdots+d_\ell$.

To bound the right hand side of \eqref{eq:dyaY}, we notice that the Fourier support of $f_\lambda$ is essentially contained in the $1$-neighborhood of $r_1S^{d_1-1}\times\cdots\times r_\ell S^{d_\ell-1}$. So we can replace $\mu_2$ by $\mu_2*\eta_{\frac{1}{r_1},\cdots,\frac{1}{r_k}}$, where $\eta_{\frac{1}{r_1},\cdots,\frac{1}{r_k}}$ is a bump function with integral $1$ essentially supported on $B^{d_1}_{\frac{1}{r_1}}\times\cdots\times B^{d_k}_{\frac{1}{r_k}}\times B^{d''}_1$. Then, by H\"older we get that
\begin{align}
&\int_{Y\times B^{d''}_1} |f_{ \lambda}(x)|^2\,d\mu_2(x) \notag\\
\leq &\bigg(\int_{Y\times B^{d''}_1} |f_{ \lambda}(x)|^p \,dx\bigg)^{\frac 2p} \bigg(\int_{Y\times B^{d''}_1} |\mu_2*\eta_{\frac{1}{r_1},\cdots,\frac{1}{r_k}}|^{\frac{p}{p-2}} \,dx\bigg)^{1-\frac 2p}\,. \label{eq:Holder}
\end{align}

To bound the second factor in \eqref{eq:Holder}, denote $x=(x_1,\cdots,x_k,x'')$ and $r_{\max}=\max\{r_1,\cdots,r_k\}$. By the choice of $\eta_{\frac{1}{r_1},\cdots,\frac{1}{r_k}}$ and the property of $\mu_2$, we get that
$$
\int_{Y\times B^{d''}_1} |\mu_2*\eta_{\frac{1}{r_1},\cdots,\frac{1}{r_k}}| \,dx \lesssim \mu_2\left(\mathcal{N}_{\frac{1}{r_1},\cdots,\frac{1}{r_k}}(Y)\times B^{d''}_1\right)
$$
and
\begin{align*}
&|\mu_2*\eta_{\frac{1}{r_1},\cdots,\frac{1}{r_k}}(x)| \lesssim r_1^{d_1}\cdots r_k^{d_k} \mu_2\left(B^{d_1}(x_1,\frac{1}{r_1})\times\cdots\times B^{d_k}(x_k,\frac{1}{r_k})\times B^{d''}(x'',1)\right)\\
& \lesssim r_1^{d_1}\cdots r_k^{d_k} \bigg(\frac{1}{r_{\max}}\bigg)^\alpha \bigg(\prod_{i=1}^k \big(\frac{r_{\max}}{r_i}\big)^{d_i}\bigg)\bigg(\prod_{i=k+1}^\ell r_{\max}^{d_i}\bigg)=r_{\max}^{d-\alpha}\leq r_1^{d-\alpha}\cdots r_k^{d-\alpha}\,.
\end{align*}
Here $\mathcal{N}_{\frac{1}{r_1},\cdots,\frac{1}{r_k}}(Y):=\left\{(y_1,\cdots,y_k): \exists (x_1,\cdots,x_k)\in Y \text{ s.t. } |x_i-y_i|<\frac{1}{r_i}, \forall i=1,\cdots,k\right\}$.
Therefore,
\begin{align}
&\bigg(\int_{Y\times B^{d''}_1} |\mu_2*\eta_{\frac{1}{r_1},\cdots,\frac{1}{r_k}}|^{\frac{p}{p-2}} \,dx\bigg)^{1-\frac 2p} \label{eq:2nd}\\
\lesssim &\bigg(\prod_{i=1}^k r_i^{(d-\alpha)\frac 2p}\bigg) \bigg(\mu_2\left(N_{\frac{1}{r_1},\cdots,\frac{1}{r_k}}(Y)\times B^{d''}_1\right)\bigg)^{1-\frac 2p}\,. \notag 
\end{align} 
 
Now we bound the first factor in \eqref{eq:Holder}. Let $x=(x',x'')\in Y\times B^{d''}_1$. For each fixed $x''$, we apply $k$-parameter refined decoupling in Theorem \ref{thm:lRefDec} (part (1)) to get that
\begin{align*}
&\bigg(\int_{Y\times B^{d''}_1} |f_{ \lambda}(x)|^p \,dx\bigg)^{\frac 2p} = \bigg(\int_{B^{d''}_1} \int_Y |f_\lambda (x',x'')|^p\,dx'\,dx''\bigg)^{\frac 2p} \\
\lessapprox & \bigg(\prod_{i=1}^k r_i^{\frac{d_i-1}{4}-\frac{d_i+1}{2p}}\bigg)M^{1-\frac 2p} \bigg(\sum_{T\in\mathbb{W}_\lambda}\int_{B^{d''}_1}\int |f_T(x',x'')|^p \,dx'\,dx''\bigg)^{\frac 2p}\\
\sim &\bigg(\prod_{i=1}^k r_i^{\frac{d_i-1}{4}-\frac{d_i+1}{2p}}\bigg) \bigg(\frac MW\bigg)^{1-\frac 2p} \bigg(\sum_{T\in\mathbb{W}_\lambda}\|f_T\|_{L^p(dx)}^2\bigg)\,,
\end{align*}
where the last inequality follows from the condition that $\|f_T\|_{L^p(dx)} \sim \lambda$ for all $T\in\mathbb{W}_\lambda$.
Furthermore, by the definition of $f_T$ we have that
\begin{align*}
&\|f_T\|_{L^p(dx)}\lessapprox |T\times B^{d''}_1|^{\frac 1p} \|f_T\|_{\infty}\\
\lesssim &\bigg|T\times B^{d''}_1\bigg|^{\frac 1p} \bigg|\vec{\sigma}_r\big(\theta(T)\times r_{k+1}S^{d_{k+1}-1}\times\cdots\times  r_\ell S^{d_\ell-1}\big)\bigg|^{\frac 12}\\
& \cdot \bigg\| \bigg[\bigg(\prod_{i=k+1}^\ell M^i_0\bigg) \bigg(M^1_{T_1}\cdots M^k_{T_k}\mu_1\bigg)\bigg]^\wedge \bigg\|_{L^2(d \vec{\sigma}_r)}\,.
\end{align*}
Hence, an orthogonality argument similarly as in \cite[(5.3)]{GIOW} implies that
\begin{align}
&\bigg(\int_{Y\times B^{d''}_1} |f_{ \lambda}(x)|^p \,dx\bigg)^{\frac 2p} \label{eq:1st}\\
\lessapprox 
&\bigg(\prod_{i=1}^k r_i^{\frac{d_i-1}{4}-\frac{d_i+1}{2p}}\bigg) \bigg(\frac MW\bigg)^{1-\frac 2p} \bigg(\prod_{i=1}^k r_i^{-(d_i-1)(\frac 12+\frac 1p)}\bigg) \bigg(\prod_{i=1}^k r_i^{-(d_i-1)} \bigg)\int |\widehat{\mu_1}|^2 \psi_r\, d\xi\,, \notag
\end{align}where the implicit constant depends on $R_0$. 

Finally, we have the following incidence estimate for the term ``$\frac MW$''. By considering the quantity
$$
\sum_{q\subset Y} \sum_{T\in\mathbb{W}_\lambda:q\cap T\neq \varnothing}\mu_2\left(\mathcal{N}_{\frac{1}{r_1},\cdots,\frac{1}{r_k}}(q)\times B^{d''}_1\right)\,,
$$
one has
$$
M\cdot \mu_2\left(\mathcal{N}_{\frac{1}{r_1},\cdots,\frac{1}{r_k}}(Y)\times B^{d''}_1\right) \lesssim W\cdot \max_{T\in \mathbb{W}_\lambda} \mu_2\left(4T_1\times\cdots \times 4T_k \times B^{d''}_1\right)\,.
$$
Then apply the $(1,\cdots,k)$-good
condition to get
\begin{equation} \label{eq:M/W}
\frac MW \lessapprox
\begin{cases}
\frac{ \prod_{i=1}^k r_i^{-(\frac{d_i}{2}-\frac{d_{\min}}{4})}}{\mu_2\left(\mathcal{N}_{\frac{1}{r_1},\cdots,\frac{1}{r_k}}(Y)\times B^{d''}_1\right)}, & d_{\min} \text{ is even}, \\
\frac{\prod_{i=1}^k r_i^{-(\frac{d_i}{2}-\frac{d_{\min}}{4}-\frac 14)}}{\mu_2\left(\mathcal{N}_{\frac{1}{r_1},\cdots,\frac{1}{r_k}}(Y)\times B^{d''}_1\right)}, &d_{\min} \text{ is odd}.
\end{cases}
\end{equation}

From the estimates
\eqref{eq:dyala}-\eqref{eq:M/W} it follows that
\begin{equation*}
\int_{E_2} |\mu_{1,g}*^{(1)}\widehat{\sigma_{r_1}}*^{(2)} \cdots *^{(\ell)} \widehat{\sigma_{r_\ell}} (x)|^2 \,d \mu_2(x) \lessapprox \bigg(\prod_{i=1}^k r_i^{\eta_i-(d_i-1)}\bigg) \int |\widehat{\mu_1}|^2 \psi_r\,d\xi\,,
\end{equation*}
where $\eta_i$ is as given in Lemma \ref{lem:MWFE}, as desired.
\end{proof}

\appendix
\numberwithin{equation}{section}
\section{Proof of Corollary \ref{cor1} and \ref{cor2}}\label{appendix: cor}
\setcounter{equation}0

First, notice that Corollary \ref{cor2} follows directly from Theorem \ref{thm: multipara proj0}. Indeed, by tracking the value of $p$ in the proof of Theorem \ref{thm: multipara proj0}, one observes that for $p\in (1,2)$,
\[
\begin{split}
&{\rm dim}\big( \{x\in\mathbb{R}^d:\, x_i\neq y_i,\, \forall y\in {\rm supp}\mu, \, i=1,\cdots,\ell, \, P_x^{(\ell)}\mu \notin L^p(S^{d_1-1}\times\cdots\times S^{d_\ell-1})\} \big)\\
\leq &2(d-1)-\alpha+\delta(p),
\end{split}
\]where $\delta(p)\to 0$ as $p\to 1$. Corollary \ref{cor2} then follows immediately. 

We now proceed to prove Corollary \ref{cor1}. Fix set $K\subset \mathbb{R}^d$ such that ${\rm dim}(K)>d-1$. We first show that
\begin{equation}\label{eqn: invisible 1}
{\rm dim}({\rm Inv}^{(\ell)}(K)\setminus {\rm Cross}(K))\leq 2(d-1)-{\rm dim}(K),
\end{equation}where
\[
{\rm Cross}(K):=\{x\in \mathbb{R}^d:\,\exists i,\, \exists y\in K, \text{ such that } x_i=y_i\}.
\]

To see this, fix any $\alpha\in (d-1, {\rm dim}(K))$, let $\mu$ be a Frostman measure supported on $K$ satisfying $\mu(B^d(x,r))\lesssim r^\alpha$ for all $x\in \mathbb{R}^d$, $r>0$. Then one observes that
\begin{equation}\label{eqn: invisible 2}
{\rm Inv}^{(\ell)}(K)\setminus {\rm Cross}(K)
\subset S^{(\ell)}(\mu),
\end{equation}recalling that
\[
\begin{split}
S^{(\ell)}(\mu)=&\{x\in \mathbb{R}^d:\, x_i\neq y_i,\, \forall y\in {\rm supp}\mu, \, i=1,\cdots,\ell,\\
&\quad P^{(\ell)}_x\mu \text{ is not absolutely continuous w.r.t. } \mathcal{H}^{d-\ell}\}.
\end{split}
\]

Indeed, for all $x\in {\rm Inv}^{(\ell)}(K)\setminus {\rm Cross}(K)$, by definition, $\mathcal{H}^{d-\ell}\left(P_x^{(\ell)}(K)\right)=0$. Hence, $P_x^{(\ell)}\mu$ must not be absolutely continuous w.r.t. $\mathcal{H}^{d-\ell}$ as $\mu$ is a probability measure.

Therefore, according to Corollary \ref{cor2},  (\ref{eqn: invisible 1}) holds true.

Next, we estimate the part ${\rm Inv}^{(\ell)}(K)\cap {\rm Cross}(K)$. Note that for any $\alpha\in (d-1, {\rm dim}(K))$, there exists a finite collection of subsets $K_1,\cdots, K_n\subset K$ such that $\mathcal{H}^\alpha(K_j)>0$, $\forall j$ and the following property holds: for all $x\in {\rm Cross}(K)$, there exists $j$ so that $x\notin {\rm Cross}(K_j)$. 

To see this, recall that we have assumed that $d_i\geq 2$, $\forall i=1,\cdots,\ell$. Hence, $\alpha>d-1> d-d_i$, $\forall i$. Therefore, according to Lemma \ref{lem: E1E2}, there exist $E_1,E_2\subset K$ with $\mathcal{H}^\alpha(E_j)>0$, $j=1,2$, satisfying that $E_1, E_2$ are well separated in each of the $\ell$ components. Applying Lemma \ref{lem: E1E2} again to $E_1$ (or $E_2$), one can find $E_{1,1}, E_{1,2}\subset E_1$ with $\mathcal{H}^\alpha(E_{1,j})>0$, $j=1,2$, and so that $E_{1,1}, E_{1,2}$ are well separated in each component too. Continuing the process, one can then find subsets $K_1, \cdots, K_{\ell+1}\subset K$ such that $\mathcal{H}^\alpha(K_j)>0$, $\forall j$, satisfying that their projections onto each $\mathbb{R}^{d_i}$ are pairwise well separated. Now, for any $x\in {\rm Cross}(K)$, each $x_i$ is in at most one of $\pi_i(K_j)$, for each $i=1,\cdots, \ell$. Therefore, as there are $\ell+1$ of the $K_j$'s in total, there must exist some $K_j$ such that $x\notin {\rm Cross}(K_j)$.

One then has that
\begin{equation}\label{eqn: invisible 3}
{\rm Inv}^{(\ell)}(K)\cap {\rm Cross}(K)\subset \bigcup_{j=1}^n \big({\rm Inv}^{(\ell)}(K_j)\setminus {\rm Cross}(K_j)\big).
\end{equation}

Indeed, suppose $x$ is not contained in the union on the right hand side. It suffices to consider the case that $x\in {\rm Cross}(K)$, as otherwise it is not contained in the left hand side as desired. Using the property above, one has that there is some $j$ such that $x\notin {\rm Cross}(K_j)$. Since $x\notin {\rm Inv}^{(\ell)}(K_j)\setminus {\rm Cross}(K_j)$, there holds $x\notin {\rm Inv}^{(\ell)}(K_j)$. This means that $K_j$ is $\ell$-parameter visible from $x$. Therefore, because $K_j\subset K$, one also has $K$ is $\ell$-parameter visible from $x$, i.e. $x\notin {\rm Inv}^{(\ell)}(K)$.

Now, applying (\ref{eqn: invisible 2}) to each $K_j$, one has from (\ref{eqn: invisible 3}) that
\[
{\rm dim}({\rm Inv}^{(\ell)}(K)\cap {\rm Cross}(K))\leq 2(d-1)-\alpha.
\]Since this holds for arbitrary $\alpha<{\rm dim}(K)$, the desired estimate follows.

\end{document}